\RequirePackage{fix-cm}
\documentclass[smallextended]{svjour3}       

\smartqed  
\usepackage{graphicx}

\usepackage {tikz}
\usepackage{pgfkeys}
\usetikzlibrary{shapes,arrows}
\usepackage{pgfplots}
\pgfplotsset{compat=1.13}
\usepackage[utf8]{inputenc}
\usepackage{fullpage,amsmath,color,xcolor,amssymb,enumitem,setspace}
\usepackage[pagebackref]{hyperref}
\usepackage{graphicx,xcolor,tikz}
\usepackage{subfigure}
\usepackage{algorithmic, algorithm}
\usepackage{xspace,color}
\usepackage{booktabs}
\usepackage{frame}
\usepackage{framed}
\usepackage{mathtools, nccmath}
\ifpdf
\DeclareGraphicsExtensions{.eps,.pdf,.png,.jpg}
\else
\DeclareGraphicsExtensions{.eps}
\fi

\hypersetup{
colorlinks   = true, 
urlcolor     = blue, 
linkcolor    = blue, 
citecolor   = red 
}

\usepackage[numbers]{natbib}

\usepackage{amsmath}
\numberwithin{theorem}{section}
\usepackage{booktabs}
\usepackage{multirow}

\allowdisplaybreaks

\usepackage{pgfkeys}
\usetikzlibrary{shapes,arrows}
\usepackage{pgfplots}
\pgfplotsset{compat=1.13}

\usepackage[todonotes={textsize=footnotesize}]{changes}
\definechangesauthor[name=AT,color=orange]{AT}
\definechangesauthor[name=MB,color=blue]{MB}

\tikzstyle{loosely dashed}=          [dash pattern=on 3pt off 6pt]
\pgfplotsset{plotOptions/.style={%
	width=\linewidth,
	xlabel={Iteration $k$},
	label style={font=\scriptsize},
	legend style={font=\scriptsize},
	xtick={1,10,100},
	tick label style={font=\scriptsize},
	solid,
	very thick
}}
\pgfplotsset{select coords between index/.style 2 args={
	x filter/.code={
		\ifnum\coordindex<#1\fi
		\ifnum\coordindex>#2\fi
	}
}}

\definecolor{silver}{cmyk}{0,0,0,0.3}
\definecolor{yellow}{cmyk}{0,.5,.3,0}
\definecolor{orange}{cmyk}{0,0.5,1,0}
\definecolor{red}{cmyk}{0,1,1,0}
\definecolor{reddishyellow}{cmyk}{0.5,0.2,0,0}
\definecolor{black}{cmyk}{0,0,0.0,1.0}
\definecolor{darkYellow}{cmyk}{0.8,0,0.0,0.5}
\definecolor{darkSilver}{cmyk}{0,0,0,0.1}

\definecolor{lightyellow}{cmyk}{0.0,0,0,0.0}
\definecolor{lighteryellow}{cmyk}{0,0,0.0,0.0}
\definecolor{lighteryellow}{cmyk}{0,0,0.0,0.0}
\definecolor{lightestyellow}{cmyk}{0.0,0,0.0,0.0}

\definecolor{colorP1}{RGB}{55,126,184}  
\definecolor{colorP2}{RGB}{228,26,28}  
\definecolor{colorP3}{RGB}{152,78,163} 
\definecolor{colorP4}{RGB}{77,175,74}  
\definecolor{colorP5}{rgb}{0.6, 0.4, 0.08} 
\definecolor{colorP6}{cmyk}{0,0.5,1,0}
\definecolor{colorP7}{RGB}{240,130,240} 

\newcommand{\bvv}{\mathbf{v}}
\newcommand{\bww}{\mathbf{w}}
\newcommand{\bee}{\mathbf{e}}

\newcommand{\bhh}{\mathbf{h}}
\newcommand{\buu}{\mathbf{u}}

\DeclareMathOperator{\dom}{dom}

\DeclareMathOperator{\ri}{ri}

\DeclareMathOperator*{\argmin}{argmin}

\DeclareMathOperator{\Fmu}{\mathcal{F}_{\mu,\infty}}
\DeclareMathOperator{\Fccp}{\mathcal{F}_{0,\infty}}
\DeclareMathOperator{\PDg}{\mathrm{PD}}
\DeclareMathOperator{\Mor}{\mathrm{M}}

\newcommand{\Rd}{\mathbb{R}^d}
\newcommand{\R}{\mathbb{R}}
\newcommand{\N}{\mathbb{N}}

\newcommand{\inner}[2]{{\langle #1; #2\rangle}}

\newcommand{\normsq}[1]{{\lVert #1\rVert ^2}}

\newcommand{\prox}{\mathrm{prox}}
\newcommand{\norm}[1]{{\lVert #1\rVert}}
\newcommand{\Expect}[2]{\ifthenelse{\equal{#2}{}}{\mathbb{E}_{{#2}\!}\left[#1\right]}{\mathbb{E}_{{#2}\!}\left[#1\right]}}

\title{Principled analyses and design of first-order methods with inexact proximal operators\thanks{MB acknowledges support from an AMX fellowship. The authors acknowledge support from the European Research Council (grant SEQUOIA 724063).This work was funded in part by the french government under management of Agence Nationale de la recherche as part of the ``Investissements d'avenir'' program, reference ANR-19-P3IA-0001 (PRAIRIE 3IA Institute).}
}

\begin{document}

\author{Mathieu Barr\'e \and Adrien B. Taylor \and Francis Bach}


\institute{Mathieu Barr\'e \at
              D.I. \'Ecole Normale Sup\'erieure, Paris, France. mathieu.barrer@inria.fr 
          \and
          Adrien B. Taylor \at
    INRIA, D.I. \'Ecole Normale Sup\'erieure, Paris, France. adrien.taylor@inria.fr
\and
            Francis Bach \at            
            INRIA, D.I. \'Ecole Normale Sup\'erieure, Paris, France. francis.bach@inria.fr
}

\date{\today}

\maketitle
\begin{abstract}
\emph{Proximal} operations are among the most common primitives appearing in both practical and theoretical (or high-level) optimization methods. This basic operation typically consists in solving an intermediary (hopefully simpler) optimization problem. In this work, we survey notions of inaccuracies that can be used when solving those intermediary optimization problems. Then, we show that worst-case guarantees for algorithms relying on such inexact proximal operations can be systematically obtained through a generic procedure based on semidefinite programming. This methodology is primarily based on the approach introduced by Drori and Teboulle~\cite{drori2014performance} and on convex interpolation results, and allows producing non-improvable worst-case analyzes. In other words, for a given algorithm, the methodology generates both worst-case certificates (i.e., proofs) and problem instances on which those bounds are achieved.
	
Relying on this methodology, we study numerical worst-case performances of a few basic methods relying on inexact proximal operations including accelerated variants, and design a variant with optimized worst-case behaviour. We further illustrate how to extend the approach to support strongly convex objectives by studying a simple relatively inexact proximal minimization method.
\end{abstract}

\let\thempfn\relax%

\section{Introduction} Proximal operations serve as base primitives in many conceptual and practical optimization methods. Formally, given a closed, proper, convex function $h:\Rd\rightarrow\R$, the proximal map of $h$, denoted by $\prox_{\lambda h}:\Rd\rightarrow\Rd$, is
\[ \prox_{\lambda h} (z)=\argmin_{x\in\Rd} \left\{\lambda h(x)+\tfrac12 \normsq{x-z} \right\},\]
where $\lambda$ is a step size.  In ideal situations, proximal operations are accessed through analytical expressions (see e.g.,~\cite{chierchia2020proximity}). However, in many cases, proximal steps have to be computed only approximately (e.g., via iterative methods). Although those problems may often be solved efficiently, one has to take those inaccuracies into account while analyzing the corresponding algorithms, in order to design methods that are sufficiently robust, and for avoiding solving the proximal subproblem to an unnecessary high precision.
Those topics are motivated in different areas of the optimization literature, in particular for augmented Lagrangian techniques (e.g., when the augmented Lagrangian has to be solved numerically), and in the context of splitting methods when proximal operators are complicated, or expensive, to compute.

In this work, we show that the performance estimation framework, originating from~\cite{drori2014performance}, can be used for studying algorithms whose base operations are approximate proximal operators. We illustrate the approach by studying numerical worst-case guarantees on various methods from the literature, and by designing an optimized inexact proximal minimization method. On the way, we survey notions of approximate proximal operators that are used in the literature.

\subsection{Motivations, contributions and organization} The main motivation of this work is to improve our capabilities of performing worst-case analyses of algorithms involving inexact proximal operations. Relying on the idea of performance estimation, and convex interpolation, we show that such analyses (i) can be completed in a principled way, and (ii) that semidefinite programming can help in the process of designing the proof. We first illustrate the approach on a variant of the inexact proximal point algorithm under a simple model of inaccuracy, and further explore the worst-case behavior of a few accelerated inexact proximal methods from~\citep{salzo2012inexact,monteiro2013accelerated}. Then, we use it for designing an optimized relatively inexact method under a generic primal-dual inaccuracy model. Finally, we use a simple inexact proximal minimization method for showing how to extend the methodology to treat strongly convex objectives.

This work is organized as follows: in Section~\ref{sec:notions} we survey common and natural notions of inaccuracies. Then, because of the structure of the inexactness criteria, we show in Section~\ref{sec:pep} that worst-case analyses of algorithms relying on such inexact proximal operations can be studied with performance estimation, which we later illustrate through several examples. Finally, we use the approach to optimize the parameters of a method relying on inexact proximal operations, in
Section~\ref{sec:orip}. Strongly convex objectives are treated in Section~\ref{sec:str-pep}, before drawing some conclusions in Section~\ref{sec:ccl}.

\subsection{Relationships with previous works}\label{sec:prevworks} Proximal operations, originally introduced by Moreau~\cite{moreau1962proximite,moreau1965proximite}, serve as base primitives in many conceptual and practical algorithms. In optimization, its use is omnipresent and originally attributed to Martinet~\cite{martinet1970breve,martinet1972det} and Rockafellar~\cite{rockafellar1976augmented,rockafellar1976monotone}.  Successful examples of algorithms relying on proximal operators include proximal gradient methods~\cite{bruck1975iterative,lions1979,passty1979ergodic,beck2009fast,nesterov2013gradient}, the celebrated alternating direction method of multipliers~\cite{Fortin,Gabay}, the related Douglas-Rachford splitting~\cite{douglas1956,lions1979,eckstein1992douglas}, and many other \emph{splitting methods}~\cite{lions1979,eckstein1989splitting}. This type of methods are abundantly used in the optimization literature, and lies at the heart of many optimization paradigms that includes distributed/decentralized optimization (e.g., through operator splitting), augmented Lagrangian techniques~\cite{rockafellar1973dual,rockafellar1976augmented,iusem1999augmented,eckstein2013practical}, and other meta-algorithms, such as ``Catalyst''~\cite{lin2015universal,lin2018catalyst}. The many aspects of their theoretical and practical uses are heavily covered
in the literature, and we defer those discussions to surveys on such topics~\cite{boyd2011distributed,combettes2011proximal,eckstein2012augmented,parikh2014proximal,ryu2016primer} and the references therein.

\paragraph{Proximal operations and inexactness} Using inexact solutions to proximal operations is not a new idea. First analyses of approximate proximal algorithms for monotone inclusions and optimization problems emerged in~\cite{rockafellar1976monotone}, and this topic appeared in many works since then (see e.g.,~\cite{guler1992new,salzo2012inexact,auslender1987numerical,solodov2001unified,fuentes2012descentwise,correa1993convergence,solodov2000error,solodov2000comparison}). Many notions of inaccuracies are also already covered in the literature. In particular, those notions were applied to the proximal point algorithm~\cite{burachik1997enlargement,eckstein1998approximate,solodov1999hybrid,monteiro2013accelerated}, inexact splitting scheme such as forward-backward splitting (and its accelerated variants)~\cite{schmidt2011convergence,villa2013accelerated,millan2019inexact,Bello2020}, Douglas-Rachford~\cite{eckstein2017approximate,Eckstein2018,svaiter2018weakly,alves2019relative}, three-operator splitting~\cite{zong2018convergence}, online optimization~\cite{dixit2019online,ajalloeian2020inexact,bastianello2020distributed}, and for designing meta-algorithms such as the hybrid approximate extragradient method~\cite{solodov1999hybrid,monteiro2010complexity,monteiro2013accelerated,alves2019inexact}, and ``Catalyst''~\cite{lin2015universal,lin2018catalyst}. Inexact proximal operations are also closely related to the theory of $\varepsilon$-subdifferentials, introduced in~\cite{brondsted1965subdifferentiability}, and to their inexact gradient and subgradient methods (see e.g.,~\cite{simonetto2016primal,millan2019inexact,devolder2013first,devolder2014first}). Finally, let us mention higher-order proximal methods, that are introduced in~\cite{nesterov2020inexactAcc,nesterov2020inexact}, and also used together with notions for approximating them.
\paragraph{Monotone inclusions} Inexact proximal methods were also studied in many works in the context of monotone operators and monotone inclusion problems~\cite{rockafellar1976monotone} (see e.g.,~\cite{bauschke2011convex} for the general topic of monotone operators, or the nice tutorial~\cite{ryu2016primer}). This was often done through notions of enlargements~\cite{burachik1998varepsilon,burachik1997enlargement,burachik2015additive}, see for example~\cite{solodov1999hybrid,solodov2001unified,burachik1999bundle,alves2019inexact,monteiro2010complexity,boct2015hybrid}. Though we are not going to work with monotone operators and inclusions, there is no apparent obstacle in applying the methodology presented here directly for dealing with inexactness in such setups.

\paragraph{Computer-assisted analyses} Using semidefinite programming for obtaining worst-case guarantees in the context of first-order optimization schemes dates back to~\cite{drori2014performance}, via so-called \emph{performance estimation problems} (PEPs), which they use to provide novel analyses of gradient, heavy-ball and accelerated gradient methods (see~\cite{polyak1964some,Nesterov:1983wy}). Performance estimation problems were coupled with ``convex interpolation'' results in~\cite{taylor2017smooth,taylor2017exact}, allowing the PEP approach to be guaranteed to generate \emph{tight} worst-case certificates. For obtaining simpler proofs, performance estimation problems can be used for designing potential functions~\cite{taylor2019stochastic}. This idea is closely related to that based on \emph{integral quadratic constraints} (IQCs), originally coined in control theory~\cite{megretski1997system}, and which were introduced for analyzing linearly-converging first-order methods in~\cite{lessard2016analysis}; and later extended to deal with sublinear convergence rates~\cite{hu2017dissipativity}. We will not further discuss IQCs here, as the current framework essentially relies on PEPs. Those methodologies being closely related, the developments below could be formulated, instead, in control-theoretic terms.

Let us mention that the PEP methodology was already taken further in different directions, as for example in the context of monotone inclusions: for the three operator splitting~\cite{ryu2018operator}, proximal point algorithm~\cite{gu2019optimal,gu2019optimal2}, and accelerated variants~\cite{kim2021accelerated}. The methodology was also used in a saddle-point setting in~\cite[Section 4.6]{drori2014contributions} and for studying worst-case properties of fixed-point iterations~\cite{lieder_halpern}. Both IQCs and PEPs were also already used for performing algorithmic design in different settings, starting through the works by~\cite{drori2014performance,kim2016optimized,drori2016optimal} and taken further in different directions~\cite{taylor2017exact,kim2018another,van2018fastest,drori2018efficient,kim2021optimizing,kim2021accelerated,ryu2019finding}. The methodology was also used in the context of multiplicative gradient noise~\cite{de2017worst2,de2017worst,cyrus2018robust}, Bregman gradient methods~\cite{dragomir2021optimal}, and adaptive first-order methods~\cite{barre2020complexity}.

\subsection{Preliminary material}\label{sec:prelim} We denote by $\Fmu$ the set of closed proper $\mu$-strongly convex functions with $0\leq\mu<\infty$, and by $\Fccp$ the corresponding subset of closed, proper and convex functions. Depending on the context, we will also use the notation $\partial h(x)$ for denoting the subdifferential of $h$ at $x$, or for abusively denoting a particular subgradient of $h$ at $x$, for notational convenience. For $h\in\Fccp$, the proximal problem can be formulated through a primal, a saddle point, or a dual formulation, as follows:
\begin{align}
&\min_x \{\Phi_p(x;z)\equiv\lambda h(x)+\tfrac12\normsq{x-z}\} \tag{P}\label{eq:prox_primal}\\
&\max_v\min_x  \{\Phi(x,v;z)\equiv\lambda  h(x)+\inner{\lambda v}{z-x}-\tfrac12\normsq{\lambda v}\}\tag{SP}\label{eq:prox_saddle}\\
&\max_v \{\Phi_d(v;z)\equiv-\lambda h^*({v} )-\tfrac12\normsq{\lambda v-z}+\tfrac12\normsq{z}\}, \tag{D}\label{eq:prox_dual}
\end{align}
where $h^*\in\Fccp$ denotes the Fenchel conjugate of $h$.
In this setting, a sufficient condition for having no duality gap is that $\ri(\dom h)\neq \emptyset$ (see e.g.,~\cite[Corollary 31.2.1]{Book:Rockafellar}, or discussions in~\cite[Section 3.5]{chambolle2016introduction}). In the following sections, we examine natural approximate optimality conditions for those three problems. Let us recall a few relations between their optimal solutions. First, first-order optimality conditions along with Fenchel conjugation allows writing
\[ x=\prox_{\lambda h}(z) \Leftrightarrow \tfrac{z-x}{\lambda}\in\partial h(x) \Leftrightarrow x\in\partial h^*(\tfrac{z-x}{\lambda})\Leftrightarrow\tfrac{z-x}{\lambda}=\prox_{{h^*}/{\lambda}}(\tfrac{z}{\lambda}).\]
By noting the last equality can be written as $\tfrac{z-\prox_{\lambda h}(x)}{\lambda}=\prox_{{h^*}/{\lambda}}(\tfrac{z}{\lambda})$, we arrive to Moreau's identity 
\begin{equation}\label{eq:moreau}
\prox_{{\lambda}h}(z)+{\lambda}\prox_{h^*/{\lambda}}(\tfrac{z}{\lambda})=z\tag{Moreau}
\end{equation} and to the corresponding identity in terms of function values: \[h(\prox_{{\lambda}h}(z))+h^*(\prox_{h^*/{\lambda}}(\tfrac{z}{\lambda}))=\inner{\prox_{{\lambda}h}(z)}{\prox_{h^*/{\lambda}}(\tfrac{z}{\lambda})}.\]
Though not being mandatory for the understanding of the material covered in the sequel, a great deal of simplifications in the exposition (particularly in the algorithmic analyses) can be obtained through the notion of $\varepsilon$-subdifferentials~\cite{brondsted1965subdifferentiability}.
\begin{definition}[Section 3 of~\cite{brondsted1965subdifferentiability}]\label{def:epssubdif} Let $h\in\Fccp(\Rd)$. For any $\varepsilon\geq 0$, we denote by $\partial_{\varepsilon}h(x)$ the $\varepsilon$-subdifferential of $h$ at $x\in\Rd$:
\begin{equation*}
\begin{aligned}
\partial_{\varepsilon}h(x)&=\{g\,|\, h(z)\geq h(x)+\inner{g}{z-x}-\varepsilon\quad \forall z\in\Rd\}\\
&=\{g\,|\,h(x)+h^*(g)-\inner{g}{x}\leq \varepsilon\}.
\end{aligned}
\end{equation*}
Any $g\in\partial_{\varepsilon}h(x)$ is called an $\varepsilon$-subgradient of $h$ at $x\in\Rd$.
\end{definition}
Before finishing this section, let us note that the symmetry of the second equality in the definition implies $g\in\partial_\varepsilon h(x)\Leftrightarrow x\in\partial_{\varepsilon}h^*(g)$.

\section{Notions of inexactness for proximal operators}\label{sec:notions} Our main motivation in this section is to survey the main natural notions of inexact proximal operations that can be used in practical applications. In particular, when solving a proximal subproblem through an iterative method, we want to be able to assess the quality of an approximate solution. Therefore, it is important to have accuracy requirements that can be evaluated in practice, and which do not depend on quantities to are generally unknown to the user, such as the exact solution to the proximal subproblem, or an optimal function value. A natural way to design such candidates accuracy conditions is to inspect optimality conditions of the proximal subproblem, and to require our approximate solutions to the subproblems to satisfy them within an appropriate accuracy. We focus on the optimization settings, but many notions extend to
the monotone operator world either directly or using concepts of enlargements~\cite{burachik1997enlargement,burachik1998varepsilon}.

Before proceeding, note that all notions do not have the same practical implications, as some might for example require having access to the dual problem~\eqref{eq:prox_dual}, or having access to $h^*$, whereas other do not. In addition, it might be easy to find approximate solutions for certain accuracy requirements, but hard to find candidates for others, depending on the target application.

In this section, we propose a list of natural notions for measuring inaccuracies within proximal operations. Those notions are not new, and our intent here is to list them in a systematic way, and to show (in the next section) that worst-case analyses of natural algorithms relying on such notions can be studied by following the same \emph{principled} steps.

Our starting point is to express optimality conditions for the proximal subproblem in its different forms~\eqref{eq:prox_primal},~\eqref{eq:prox_saddle}, and~\eqref{eq:prox_dual}, as follows.
\begin{itemize}
    \item First-order optimality conditions of~\eqref{eq:prox_saddle} can be written as
    \[ \left\{\begin{array}{l}
    x=\prox_{\lambda h}(z)       \\
    v=\prox_{h^*/\lambda}(\tfrac{z}{\lambda}) 
    \end{array}\right.\Leftrightarrow 0\in\begin{pmatrix}\partial_x\Phi(x,v;z)\\ \partial_v(-\Phi(x,v;z))\end{pmatrix},\]
    which can equivalently be formulated as the optimality conditions of either~\eqref{eq:prox_primal} or~\eqref{eq:prox_dual}:
    \[ \left\{\begin{array}{l}
    0=\norm{w-v} \text{ for some } w\in\partial h(x) \text{ $\Leftrightarrow$ } 0=\norm{u-x} \text{ for some } u\in\partial h^*(v),  \\
    0=\norm{x-z+\lambda v}.
    \end{array}\right.\]
    \item Assuming no duality gap occurs between~\eqref{eq:prox_primal} and~\eqref{eq:prox_dual} (see Section~\ref{sec:prelim}), one can write the zeroth-order optimality conditions (i.e., the primal-dual gap) for~\eqref{eq:prox_saddle}
    \[ \left\{\begin{array}{l}
    x=\prox_{\lambda h}(z)       \\
    v=\prox_{h^*/\lambda}(\tfrac{z}{\lambda}) 
    \end{array}\right. \Leftrightarrow \Phi_p(x;z) - \Phi_d(v;z) = 0,\]
    which can explicitly be written as
    \[ \Phi_p(x;z) - \Phi_d(v;z)= \lambda h(x)+\lambda h^*({v} )- \lambda\inner{x}{v} + \tfrac{1}{2}\|x -z +\lambda v\|^2.\]
    We observe in the previous primal-dual gap expression that it decomposes as the sum of two nonnegative quantities $\lambda h(x)+\lambda h^*({v} )- \lambda\inner{x}{v}$ and $\tfrac{1}{2}\|x -z +\lambda v\|^2$.
    In particular, the first term controls how far is $v$ from $\partial h(x)$. Indeed, first-order optimality conditions applied to the definition of the Fenchel-Legendre transform (see e.g.,~\cite[Theorem 23.5]{Book:Rockafellar}) gives 
    \[0=\lambda h(x)+\lambda h^*({v} )- \lambda\inner{x}{v}\Leftrightarrow v\in\partial h(x)\Leftrightarrow x\in\partial h^*(v).\]
    Moreover, when this term is nonzero, one can express the relationship between $x$ and $v$ through $\varepsilon$-subdifferentials (see Definition~\ref{def:epssubdif}) as
    \[h(x)+ h^*({v} )- \inner{v}{x}\leq\varepsilon\Leftrightarrow v\in\partial_{\varepsilon}h(x)\Leftrightarrow x\in\partial_{\varepsilon}h^*(v).\]
    In other word, for any primal-dual pair $(x,v)$, $v$ is always an $\varepsilon$-subgradient of $h$ at $x$ with $\varepsilon = h(x) +h^*(v) - \inner{x}{v}$ (which is finite when $v\in\dom h^*$).
    \end{itemize}
    Those elements motivate measuring inaccuracies simultaneously in two ways:
    \begin{itemize}
        \item[(i)] requiring $\norm{x-z+\lambda v}$ being small enough---i.e., requiring~\eqref{eq:moreau} to hold approximately---, and
        \item[(ii)] requiring either $v$ being close enough to $\partial h(x)$, and/or how $x$ being close enough to $\partial h^*(v)$. Via the primal-dual gap formulation, this is done by requiring $h(x)+ h^*({v} )- \inner{v}{x}$ to be small enough. In first-order optimality conditions, this could be done by requiring $\norm{v-w}$ to be small enough for some $w\in\partial h(x)$ or $\norm{x-u}$ to be small enough for some $u\in\partial h^*(y)$.
    \end{itemize}
    Note that when either the candidate dual solution satisfies $v\in\partial h(x)$, or the candidate primal solution satisfies $x\in\partial h^*(v)$ (for example if the proximal subproblem is solved via a purely primal, or purely dual, method), then the only term that needs to be controlled is that of (i). In the case where either the approximate dual solution is chosen as $v=\tfrac{z-x}{\lambda}$ or the approximate primal solution is chosen as $x={z-\lambda v}$, the only term that needs to be controlled is~(ii), as~(i) is automatically $0$. In other cases, both terms need to be controlled.
  
    \subsection{A few observable notions of inexactness} In this section, we are interested in inexactness notions that do no require knowledge on $\prox_{\lambda h}(z)$ or $\prox_{h^*/\lambda}(\tfrac{z}{\lambda})$ to be evaluated. In what follows, we denote the primal-dual gap by
    \[ \PDg_{\lambda h}(x,v;z) := \Phi_p(x;z) - \Phi_d(v;z),\]
    and the Moreau gap by
    \[ \Mor_\lambda(x,v;z):=\normsq{x-z+\lambda v},\]
    for convenience, and we recall a property on the primal-dual gap that was stated earlier in Section~\ref{sec:notions} but that is key to compare it with $\varepsilon$-subgradient based criterion in the literature.
    \begin{lemma}\label{lem:pdgag-eps}
        Let $\varepsilon\geq0$, $x,v,z \in \Rd$. If $v \in \partial_\varepsilon h(x)$, then the following inequality holds
        \[\PDg_{\lambda h}(x,v;z) \leq \tfrac{1}{2}\Mor_{\lambda}(x,v;z) + \lambda\varepsilon.\]
        Furthermore, it holds with equality when $\varepsilon = h(x)+h^*(v)-\inner{x}{v}$.
        
        Reciprocally, let $\sigma \geq 0$, $x,v,z\in\Rd$, if $\PDg_{\lambda h}(x,v;z) \leq \sigma$ then,
        \[ v \in \partial_{\varepsilon_v}h(x), \quad \text{with } \varepsilon_v = \tfrac{\sigma}{\lambda} -\tfrac{1}{2\lambda}\Mor_{\lambda}(x,y;z).  \]
    \end{lemma}
    Therefore, imposing an upper bound on the right hand side, automatically imposes a bound on the primal-dual gap.
    We list a series of criterion that were used in different works for quantifying the quality of some primal-dual pair $(x,v)$ for approximating the pair $(\prox_{\lambda h}(z), \prox_{h^*/\lambda}(z/\lambda))$. In all the criteria that follow, $\sigma$ denotes an error magnitude that we do not specify for now as we focus on the left hand side of the inexactness criteria. 
\begin{itemize}
    \item (Primal-dual inaccuracy, take I) The quality of a primal-dual pair $(x,v)$ for approximating the couple $(\prox_{\lambda h}(z), \prox_{h^*/\lambda}(z/\lambda))$ can be monitored by requiring
    \[ \PDg_{\lambda h}(x,v;z) \leq \sigma,\]
    to hold for some predefined $\sigma\geq 0$. Using Lemma~\ref{lem:pdgag-eps}, one can reformulate this requirement as $\exists \varepsilon\geq 0$: $v\in\partial_\varepsilon h(x)$ and $\tfrac12\normsq{x-z+\lambda v}+\lambda \varepsilon \leq \sigma$. This criterion is used among others in the hybrid approximate extragradient (HPE) framework~\cite{solodov1999hybrid,solodov2000comparison,solodov2000error,solodov2001unified}, in its inertial/accelerated versions~\cite{monteiro2013accelerated,boct2015hybrid,alves2019inexact}, or for forward-backward splittings~\cite{millan2019inexact,Bello2020}. This criterion is generalized in the (monotone) operator world, through the notion of $\varepsilon$-enlargements~\cite{burachik1998varepsilon,burachik1997enlargement}, generalizing the notion of $\varepsilon$-subdifferentials.
\end{itemize}
 Stronger notions of primal-dual pairs can be obtained by coupling the primal and dual estimates, as follows.
 \begin{itemize}
    \item (Primal-dual inaccuracy, take II) The quality of a primal point $x$ for approximating $\prox_{\lambda h}(z)$ can be monitored by constructing an approximate dual point through~\eqref{eq:moreau}: $v=\tfrac{z-x}{\lambda}$ and requiring the corresponding primal-dual gap to satisfy
    \[ \PDg_{\lambda h}(x,\tfrac{z-x}{\lambda};z) \leq \sigma.\]
    Note that this formulation can be rewritten as $\PDg_{\lambda h}(x,\tfrac{z-x}{\lambda};z) = \lambda h(x) + \lambda h^*(\tfrac{z-x}{\lambda}) - \lambda\inner{x}{\tfrac{x-z}{\lambda}} \leq \sigma \Leftrightarrow \tfrac{z-x}{\lambda}\in \partial_{\sigma/\lambda} h(x)$, or equivalently $x = z - \lambda v$ with $v \in \partial_{\sigma/\lambda} h(x)$, or even in a dual form $v=\tfrac{z-u}{\lambda}$ with $u\in\partial_{\sigma/\lambda} h^*(v)$. This notion of inaccuracy was also used in quite a few works, see e.g.,~\cite{lemaire1992convergence,cominetti1997coupling} and more recently in~\cite{villa2013accelerated} and~\cite[``approximation of type 2'']{salzo2012inexact}.
    \item (Primal-dual inaccuracy, take III) The quality of a primal point $x$ for approximating $\prox_{\lambda h}(z)$ can be monitored by constructing an approximate dual point as $v=h'(x)\in\partial h(x)$ and by requiring
    \[\PDg_{\lambda h}(x,h'(x);z)\leq \sigma.\]
    In this case, the criterion can be written as $\PDg_{\lambda h}(x,h'(x);z) = \tfrac12\|x-z+\lambda h'(x)\|^2 \leq \sigma$, which is equivalent to $x = z - \lambda v + \lambda e$ with $v\in\partial h(x)$ and $\tfrac{\lambda^2}{2}\normsq{e}\leq \sigma$. This error criterion was among the first to be used, see~\cite{rockafellar1976monotone}, and was later used in many works, see e.g.,~\cite{burke1999variable,solodov1999hybridproj,solodov2000comparison,solodov2000error,eckstein1998approximate,alves2019relative}, and~\cite[``approximation of type 3'']{salzo2012inexact}.
 \end{itemize}
Among known methods for dealing with inexact proximal iterations, extra-gradient methods occupy an important place (see, e.g., the conceptual algorithm in~\cite{nemirovski2004prox}). Intuitively, the idea is to compute some intermediate point $u\approx \prox_{\lambda h}(z)$, to evaluate some $u'\in\partial h(u)$ (or an $\epsilon$-subgradient version of it), and to use $x=z-\lambda u'$ as our working approximation of $\prox_{\lambda h}(z)$. Natural notions of inaccuracy applied on $u$ can also then directly be interpreted in terms of $x$, as follows.
\begin{itemize}
    \item (Primal-dual inaccuracy, take IV) One way to interpret the hybrid proximal extra-gradient method~\cite{solodov1999hybrid} is that it measures the quality of a primal point $x$ for approximating $\prox_{\lambda h}(z)$ by requiring the existence of some other primal point $u$ satisfying
    \[ \PDg_{\lambda h}(u,\tfrac{z-x}{\lambda};z)\leq\sigma.\]
    Equivalently, one can write this condition as $\exists \varepsilon\geq 0$ and $\exists u\in\partial_\varepsilon h^*(\tfrac{z-x}{\lambda})$ such that $\tfrac12 \normsq{u-x}+\lambda \varepsilon \leq \sigma$, which we can also explicitly rewrite in an extra-gradient format as: $x = z - \lambda u'$ with $u' \in \partial_\varepsilon h(u) , \tfrac{1}{2}\|u-z+\lambda u'\|^2 + \lambda \varepsilon \leq \sigma$ for some feasible $u$. In other words, it corresponds to obtain a $u\approx\prox_{\lambda h}(z)$ according to the primal-dual inaccuracy criterion (take I) on $u$, and to use $x=z-\lambda u'$ as the working approximation of $\prox_{\lambda h}(z)$.
    \item (Primal-dual inaccuracy, take V) A stronger version of the previous construction for measuring inaccuracy of $x$ consists in picking $u\in \partial h^*(\tfrac{z-x}{\lambda})$ and requiring
    \[\PDg_{\lambda h}(u,\tfrac{z-x}{\lambda};z) = \tfrac12\|x - u\|^2 \leq \sigma.\]
    In this setting, one can rewrite $x = z - \lambda u'$ with $u'\in\partial h(u)$ with $\tfrac12\|x-u\|^2 \leq \sigma$. This condition was presented, and used, in~\cite{solodov2000inexact} (though not exactly using this viewpoint). This corresponds to apply the primal-dual inaccuracy criterion (take III) on $u\approx\prox_{\lambda h}(z)$, and to use $x=z-\lambda u'$ as the working approximation of $\prox_{\lambda h}(z)$. This criterion is also used in~\cite{eckstein2017approximate} for relatively inexact Douglas-Rachford and ADMM. 
\end{itemize}
Perhaps curiously, applying the same extra-gradient idea to primal-dual inaccuracy (take II), one recovers (take II) without any change.

One can then do the same exercise by requiring first-order optimality conditions to be approximatively satisfied. As previously explained, the corresponding notions of inexactness actually collapse with those based on primal-dual requirements as soon as either the dual variable is a subgradient of $h$ at the primal point $v\in\partial h(x)$, or equivalently when $x\in\partial h^*(v)$.
\begin{itemize}
    \item (Primal-dual subgradient residual) Among the many possibilities for quantifying the quality of a primal-dual pair $(x,v)$ as an approximation of the solution $(\prox_{\lambda h}(z),\prox_{h^*/\lambda}(z/\lambda))$, one probably natural criterion is to require
    \begin{equation*}\label{eq:inexact_max}
    \max\{ \norm{x-z+\lambda v},\,\norm{v-\partial h(x)},\,\norm{x-\partial (h^*)(v)}\}\leq \sigma.     
    \end{equation*} Another possibility is to require a positively weighted sum of those different terms to be small enough.
\end{itemize}
Note that one can design alternate criteria by performing conic combinations, intersections and unions of previous inaccuracy criteria. Finally, note that the choice of the most appropriate criterion depend on the application at hand (e.g., depending on the cost of obtaining an approximation satisfying the accuracy requirement, and on the cost of checking it).

\begin{remark}\label{rem:practical}
In practice, as soon as one can use a first-order (or higher-order) method for solving (P), (D) or (SP) there are often different ways to obtain primal-dual pairs (x,v) satisfying some primal-dual inexactness requirement. Depending on the application, $h^*$ and $\partial h^*$ might or might not be available, rendering some criterion irrelevant for that particular application.
In particular, it is common that \eqref{eq:prox_primal} can be solved approximately and one has access to elements of $\partial h (x)$. Then, criteria of the form $\PDg_{\lambda h}(x,\partial h(x);z) \leq \sigma$ can be used, as in \cite{alves2019relative}.
\end{remark}    
    
\subsection{Abstract, generally non-observable, notions of inexactness}
Some notions are more complicated to directly monitor in practice. However, they might allow modeling certain situations that are not covered by previous notions (such as dealing with possibly infeasible primal and dual solutions).
\begin{itemize}
	\item (Purely primal (or dual) inaccuracy) One can monitor the quality of 	
	an approximate $x\approx\prox_{\lambda h} (z)$ by requiring $x$ to satisfy, for some $\sigma\geq 0$ 
	\begin{equation*}
	\PDg_{\lambda h}(x,\prox_{{h^*}/{\lambda}}(\tfrac{z}{\lambda});z)=\Phi_p(x;z) - \Phi_p(\prox_{\lambda h}(z);z)\leq \sigma
	\end{equation*}
	This notion is directly considered, e.g., in~\cite{auslender1987numerical,schmidt2011convergence,lin2015universal,lin2018catalyst}, in~\cite[``approximation of type 1'']{salzo2012inexact}, and indirectly in other works (e.g.,~\cite[Lemma 3.1]{guler1992new}). Although it is mostly impractical (as it requires knowing the optimal value of the proximal subproblem), it can be verified indirectly via other error criterion (such as a primal-dual gap). In the same spirit, one could use purely dual requirements $\Phi_d(\prox_{{h^*}/{\lambda}}(z/\lambda);z) - \Phi_d(v;z)$.
	\item (Distance to the solution) A primal candidate $x\approx\prox_{\lambda h} (z)$ may be required to be close to $\prox_{\lambda h} (z)$. That is, for some $\sigma>0$, one may require
	\begin{equation*}
	\norm{x-\prox_{\lambda h}(z)}\leq\sigma.
	\end{equation*}
	Note that it corresponds to verify an approximate Moreau gap $\Mor_{\lambda}(x,v;z)$ with $v=\prox_{h^*/\lambda}(z/\lambda)$.
	This notion is also not new~\cite{rockafellar1976monotone,guler1992new}, and can also be verified indirectly, e.g., via $\tfrac12 \normsq{x-\prox_{\lambda h}(z)}\leq \PDg_{\lambda h}(x,\prox_{{h^*}/{\lambda}}(\tfrac{z}{\lambda});z)$.
	Its dual version $\norm{\prox_{h^*/\lambda}(z/\lambda) -v}$, or primal-dual notion $\normsq{x-\prox_{\lambda h}(z)}$ $+\lambda^2\normsq{\prox_{h^*/\lambda}(z/\lambda) -v}$  could also be considered.
\end{itemize}
\subsection{Absolute versus relative inaccuracies} Depending on algorithmic requirements, error tolerances might be specified in terms of absolute constants, or as functions of the state of the algorithm at hand. For example, a common situation is to choose some absolute constant $\sigma>0$, and to require $\PDg_{\lambda h}(x,v;z)\leq \sigma$,
where $\sigma$ should typically be chosen as a decreasing function of the iteration counter. A standard alternative is to pick a relative type of accuracy requirement, such as $\PDg_{\lambda h}(x,v;z)\leq \norm{x-z}^2$.
Both types of requirements are pretty standard, and were already stated in early developments on inexact proximal methods (see e.g.,~\cite[condition (A) or (B)]{rockafellar1976monotone}. Relative versions often offer the advantage of being simpler to tune, sometimes at the cost of worse performances, see e.g.,~\cite{lin2018catalyst}.

\section{Principled, and computer-assisted worst-case analyses}\label{sec:pep}
In this section, we show that a generic inexact proximal method can be analyzed using performance estimation problems. Those problems were introduced in~\cite{drori2014performance} for analyzing fixed-step first-order methods for smooth convex optimization, and were extended in a few directions since then, see \S``Computer-assisted analyses'' in Section~\ref{sec:prevworks}. 

In short, we provide a principled approach to obtain rigorous worst-case guarantees and the corresponding proofs for a class of inexact proximal methods. The idea is to formulate the problem of performing a worst-case analysis as an optimization problem, which can be solved numerically. Feasible points to this problem correspond to matching examples (i.e., worst-case instances: functions and iterates) and feasible points to the dual problem correspond to worst-case guarantees (i.e., proofs). The possibility of solving those problems numerically essentially allows \emph{sampling} worst-case examples and proofs for given problems and algorithmic parameters (for instance, step sizes and accuracy levels).

\subsection{A class of inexact proximal methods} In this section we consider the minimization problem
\[ \min_{x\in\Rd}h(x)\]
with $h\in\Fccp$ (a closed, proper, and convex function) and define a class of approximate proximal methods for solving this problem, along with a principled way of analyzing them.

\subsubsection{Fixed-step inexact proximal methods} Let $x_0\in\Rd$  be an initial point, and let $\{\lambda_{i}\}_{i}$ be a sequence of nonnegative step sizes. When exact proximal computations are available, 
a natural class of methods can be described by
\[    w_{k+1} = \prox_{\lambda_{k+1}h}\left(w_k - \sum_{i=1}^{k}\alpha_{k+1,i}v_i\right)\]
where $v_i \in \partial h(w_i)$ for $i=1,\hdots,k$ and $\{\alpha_{i,j}\}_{ij}$ is a sequence of parameters. 
In this setting, the next iterate of the method is obtained as the result of the proximal operator of $h$ applied to the previous iterate plus a linear combination of previously encountered subgradients. It can be reformulated as
\[    w_{k+1} = w_k - \sum_{i=1}^{k}\alpha_{k+1,i}v_i - \lambda_{k+1}v_{k+1}\]
where $v_{k+1}\in\partial h(w_{k+1})$, which corresponds to optimality conditions of the proximal subproblems. 

We extend this class of algorithms for inexact proximal computations by introducing some error terms $\{e_i\}_i$ in the previous formulation.
\[    w_{k+1} \approx \prox_{\lambda_{k+1}h}\left(w_k - \sum_{i=1}^{k}\alpha_{k+1,i}v_i -\sum_{i=0}^k\beta_{k+1,i}e_i\right)\]
where $v_i \in \partial h(w_i)$ for $i=1,\hdots,k$ and $\{\alpha_{i,j}\}_{ij}$, $\{\beta_{i,j}\}_{ij}$ are sequences of parameters. In particular, $\{\beta_{i,j}\}_{ij}$ allows the method to take into account the errors made in previous proximal computations. We disambiguate the $\approx$ notation by introducing an additional error term $e_{k+1}$ and define the class of \emph{fixed-step inexact proximal methods} as 
\begin{equation}\label{eq:generic-prox}
    w_{k+1} = w_k - \sum_{i=1}^{k}\alpha_{k+1,i}v_i -\sum_{i=0}^k\beta_{k+1,i}e_i - \lambda_{k+1}(v_{k+1}+e_{k+1}),
\end{equation}
where $v_{k+1}\in \partial h(w_{k+1})$. The error source in the proximal operation comes from the fact that $v_{k}+e_{k}$ does not necessarily belong to $\partial h(w_{k})$.  
For modelling the error incurred in the proximal operations, in particular the discrepancy between $v_k+e_k$ an $\partial h(w_k)$, we are allowed to use all notions from previous sections. We abstract this modelling
step by imposing on the iterates some (possibly vector) inequalities of the form
 \begin{equation}\label{eq:generic_inexact}
    \text{EQ}_{k}(w_0,\hdots,w_k,v_0,\hdots,v_k,e_0,\hdots,e_k,h(w_0),\hdots,h(w_k))\leq 0.
\end{equation}
For readability purposes, we abusively use $\text{EQ}_k$ without explicitly instantiating the inputs in what follows.

In addition, all the inexactness criteria of Section~\ref{sec:notions} share a common structure which we refer to as ``Gram-representable", as follows.
\begin{definition}
A criterion \eqref{eq:generic_inexact} is \textbf{Gram-representable} if it is affine in $h(w_0),$ $\hdots,$ $h(w_k)$ and in $\inner{x}{y}$ for all $x,y\in\{w_i\}_{i\in[0,k]}\cup\{v_i\}_{i\in[0,k]}\cup\{e_i\}_{i\in[0,k]}$.
\end{definition}
All methods in the form \eqref{eq:generic-prox} subject to Gram-representable \eqref{eq:generic_inexact} can be analyzed in a principled way using the performance estimation procedure presented in the next section. Furthermore, all inaccuracy criterion presented in Section~\ref{sec:notions} are actually Gram-representable.

\subsubsection{Examples}\label{sss:examples_algo} Before going into the analyses, let us provide a few examples of methods that fit into model~\eqref{eq:generic-prox} with Gram-representable models of the form~\eqref{eq:generic_inexact}. In all cases, we let $\{\lambda_k\}_{k}$ be a sequence of predefined step sizes.
\begin{itemize}
    \item The vanilla \emph{proximal minimization algorithm} is given by
    \[ x_{k+1}=x_k-\lambda_{k+1}v_{k+1}, \]
    with $v_{k+1}\in\partial h(x_{k+1})$. It fits in~\eqref{eq:generic-prox} with $\alpha_{i,j}=\beta_{i,j}=0$, as well as $e_k=0$ which can be transcribed into a Gram-representable \eqref{eq:generic_inexact}.
    
    \item The \emph{inexact proximal minimization algorithm} proposed in~\cite[Section 3]{rockafellar1976augmented} can be described by
    \[ x_{k+1}=x_k-\lambda_{k+1}(v_{k+1}+e_{k+1}), \]
    with $v_{k+1}\in\partial h(x_{k+1})$, with the error term $e_{k+1}$ being controlled via either \[\|e_{k+1}\|^2\leq\tfrac{\epsilon_{k+1}^2}{\lambda_{k+1}^2}, \quad \text{ or }\quad \|e_{k+1}\|^2\leq\tfrac{\delta_{k+1}^2}{\lambda_{k+1}^2}\|x_{k+1}-x_{k}\|^2\]  for some appropriate sequence $\{\epsilon_k\}_k$~\cite[Criterion (A')]{rockafellar1976augmented}, or $\{\delta_k\}_{k}$~\cite[Criterion (B')]{rockafellar1976augmented}. In both cases, the method fits into model~\eqref{eq:generic-prox} with $\alpha_{i,j}=\beta_{i,j}=0$ and a Gram-representable~\eqref{eq:generic_inexact}. Depending on how we decide to control the error, we can either pick $\text
    {EQ}_{k+1}=\normsq{e_{k+1}}-\tfrac{\varepsilon_{k+1}^2}{\lambda_{k+1}^2}$ or $\text
    {EQ}_{k+1}=\normsq{e_{k+1}}-\tfrac{\delta_{k+1}^2}{\lambda_{k+1}^2}\normsq{x_{k+1}-x_{k}}$. 
\end{itemize}
Many known proximal methods rely on using the past first-order information for improving convergence guarantees of the sequence iterates.
\begin{itemize}
    \item \emph{Güler proximal point algorithm} \cite[Section 6]{guler1992new} is defined as follow given $\beta_0 >0$, $y_0 = x_0\in \Rd$ and $\{\lambda_k\}_{k}$ a sequence of positive step sizes
    \[\left \{ \begin{array}{rcl}
         t_{k+1}&=& \tfrac{1+\sqrt{1+4t_k^2}}{2}  \\
         x_{k+1} & = & y_{k} - \lambda_{k+1} v_{k+1} \text{ with } v_{k+1} \in \partial h (x_{k+1})\\
         y_{k+1} & = & x_{k+1} + \tfrac{t_k-1}{t_{k+1}}(x_{k+1}-x_k) + \tfrac{t_k}{t_{k+1}}(x_{k+1}-y_k)
    \end{array}\right. \]
    One can substitute the $y_{k+1}$ by $x_{k+2}+\lambda_{k+2}v_{k+2}$ and $y_k$ by $x_{k+1}+\lambda_{k+1}v_{k+1}$ in the last definition, which leads to
    \[ x_{k+2} = \left(1 + \tfrac{t_k-1}{t_{k+1}} \right)x_{k+1} -\tfrac{t_k-1}{t_{k+1}}x_k - \tfrac{t_k\lambda_{k+1}}{t_{k+1}}v_{k+1} - \lambda_{k+2}v_{k+2}. \]
    This allows to show recursively that the $\{x_k\}_k$ belong to the class \eqref{eq:generic-prox}. Indeed $x_1 = x_0 - \lambda_1v_1$. Then suppose that $x_{k+1}$ has the form of \eqref{eq:generic-prox}, with $\beta_{k+1,i} = 0$ and $e_{k+1} =0$, then 
    \begin{align*}
        x_{k+2} =& x_{k+1} -\tfrac{t_k-1}{t_{k+1}}\left(\sum_{i=1}^k\alpha_{k+1,i}v_i+ \lambda_{k+1}v_{k+1} \right)- \tfrac{t_k\lambda_{k+1}}{t_{k+1}}v_{k+1} - \lambda_{k+2}v_{k+2}.
    \end{align*} 
    And we can identify $\alpha_{k+2,i} = \tfrac{t_k-1}{t_{k+1}}\alpha_{k+1,i}$ for $i=1\hdots k$ and $\alpha_{k+2,k+1} = \tfrac{t_k-1}{t_{k+1}}\lambda_{k+1}$, as well as $\beta_{k+2,i}=0$ and $e_{k+2}=0$. 
\end{itemize}
Other methods that fit in \eqref{eq:generic-prox} with Gram-representable inexactness criterion \eqref{eq:generic_inexact} include the \emph{hybrid approximate extragradient algorithm}~\cite{solodov1999hybrid} (details in Appendix~\ref{app:methods}), the \emph{inexact accelerated proximal point algorithm}~\emph{IAPPA1} and \emph{IAPPA2}~\cite{salzo2012inexact} (details in Appendix~\ref{app:methods}), \emph{A-HPE}~\cite{monteiro2013accelerated} (see details in Appendix~\ref{app:methods}), and \emph{Catalyst} \cite{lin2015universal}.

\subsection{Computing worst-case guarantees} In this section, we provide a principled approach for performing worst-case analyses of fixed-step inexact proximal methods written in terms of~\eqref{eq:generic-prox} and~\eqref{eq:generic_inexact}. Let $N\in \N$ and $R \in \R^*$, for simplicity of the exposition, we only consider worst-case guarantees of type
\begin{equation}\label{eq:wc_type} h(w_N)-h(w_\star)\leq C(N,R),\end{equation}
for all $h\in\Fccp(\Rd)$, $w_\star\in\argmin_x h(x)$, $w_0 \in \Rd$ such that $\normsq{w_0-w_\star}\leq R^2$, and $d\in\mathbb{N}$. Our  goal is then to compute values of $C(N,R)$, hopefully small and decreasing with $N$, for this inequality to be valid. This choice is made for simplicity purposes, and can be changed (e.g. Section \ref{sec:str-pep}); see discussions and examples in~\cite{taylor2017exact,taylor2017performance}.

Given a method in the form~\eqref{eq:generic-prox} (i.e., fixed $\{\alpha_{i,j}\}_{ij}$, $\{\beta_{i,j}\}_{ij}$) as well as inexactness criteria in the form \eqref{eq:generic_inexact} (i.e., fixed $\{\text{EQ}_i\}_i$), we formulate the problem of computing the smallest $C(N,R)$ such that~\eqref{eq:wc_type} is valid. For doing that, we look for the worst problem instance for guarantees of type~\eqref{eq:wc_type}, that is, a convex function on which $h(w_N)-h(w_\star)$ is the largest possible when $\normsq{w_0-w_\star}\leq R^2$
\begin{equation}\label{eq:pep}
\begin{aligned}
C(N,R)\geq\max_{\substack{d,h\\w_\star,w_0,\hdots,w_N\in\Rd\\v_0,\hdots,v_{N}\in\Rd\\e_0,\hdots,e_{N}\in\Rd}}& h(w_N)-h(w_\star)\\
\text{s.t. }& h\in\Fccp(\Rd), \quad w_\star\in\argmin_x h(x)\\
&\normsq{w_0-w_\star} \leq R^2\\
&w_1,\hdots,w_N \text{ satisfying } \eqref{eq:generic-prox}\\
&\text{EQ}_{k} \leq 0 \quad k=0,\hdots,N.
\end{aligned}
\end{equation}
This type of problems is often referred to as a \emph{performance estimation problem} (introduced in~\cite{drori2014performance}). It is intrinsically infinite dimensional, as it contains a variable $h\in\Fccp$. One possible way of dealing with this variable is to restrict ourselves to work with a discrete (or sampled) version of $h$. For doing that, we introduce a set $S$ containing sampled points of $h$, in the form $S=\{(w_i,v_i,h_i)\}_{i}$, and we reformulate the previous problem using the requirement $h_i=h(w_i)$, $v_i\in\partial h(w_i)$. In addition, \eqref{eq:generic_inexact} implies that the $\text{EQ}_k$ are only described using $\{e_i\}_i$ and the elements of $S$ (we emphasize this by writing $\text{EQ}_k(S,e)$), thus we can write
\begin{equation}\label{eq:pep_sampled}
\begin{aligned}
C(N,R)\geq\max_{\substack{d\\\ S\subset\Rd\times \Rd\times\mathbb{R}\\e_0,\hdots,e_{N}\in\Rd }}& h_N-h_\star&\\
\text{s.t. }&S=\{(w_i,v_i,h_i)\}_{i\in \{\star,0,1,\hdots,N\}} \\
&\exists h\in\Fccp: \, f=h(x),\, g\in\partial h(x) \quad \forall (x,g,f)\in S\\ 
&v_\star =0, \quad\normsq{w_0-w_\star} \leq R^2\\
&w_1,\hdots,w_N \text{ satisfying } \eqref{eq:generic-prox}\\
&\text{EQ}_{k}(S,e) \leq 0 \quad k=0,\hdots,N.
\end{aligned}
\end{equation}


Now, a key step is to rely on interpolation (also often referred to as extension) theorems for formulating the existence constraints in a tractable way. Such results can be formulated as follows (see e.g.,~\cite[Theorem 1]{taylor2017smooth}) :
\begin{equation}\label{eq:interp}
\begin{aligned}
\exists h\in\Fccp: \, f=h(x),\quad &g\in\partial h(x) \quad \forall (x,g,f)\in S \\ &\Leftrightarrow f'\geq f+\inner{g}{x'-x} \quad \forall (x,g,f),(x',g',f')\in S.
\end{aligned}
\end{equation}
It allows arriving to a nearly quadratic problem (still dependent on a dimension variable $d$).
\begin{equation}\label{eq:pep-ppa}
\begin{aligned}
\max_{\substack{d\\\ S\subset\Rd\times \Rd\times\mathbb{R}\\e_0,\hdots,e_{N}\in\Rd }}& {h_N-h_\star}&\\
\text{s.t. }&S=\{(w_i,v_i,h_i)\}_{i\in \{\star,0,1,\hdots,N\}}\\
&f'\geq f+\inner{g}{x'-x} \quad \forall (x,g,f),(x',g',f')\in S\\ &v_\star=0,\quad \normsq{w_0-w_\star}\leq R^2\\
&w_1,\hdots,w_N \text{ satisfying } \eqref{eq:generic-prox}\\
&\text{EQ}_{k}(S,e) \leq 0  \quad k=0,\hdots,N.
\end{aligned}
\end{equation}

\begin{remark}\label{rem:epssubdiff}Let us note that inexactness requirements for proximal operators are often formulated through $\varepsilon$-subdifferentials. In order to simplify the performance estimation problems, one can use appropriate interpolation conditions for directly incorporating $\varepsilon$-subdifferentials. Since this interpolation result is rather a trivial extension of regular convex interpolation (see e.g.,~\cite[Theorem 1]{taylor2017smooth}), we provide it in Appendix~\ref{app:epssubdiff}.
\end{remark}

The next section presents how problem \eqref{eq:pep-ppa} can be reformulated as linear semidefinite program when the $\text{EQ}_k$ are Gram-representable.

\subsection{Semidefinite formulation}\label{ss:sdp}
Let
\[H = [h_0-h_\star,\,\,\,\,h_1-h_\star,\,\,\,\hdots\,\,\,h_N-h_\star] \in \R^{1\times(N+1)}\] 
a flat vector containing function values and 
\begin{align*}
G &= X^TX \succeq 0 \text{ with }\\
 X &= [w_\star,\,\,w_0,\,\,v_0,\,\,\hdots\,\,v_N,\,\,e_0,\,\,\hdots\,\,e_N] \in \R^{d\times (2N+4)}
\end{align*}    
a Gram matrix of the vector variables of \eqref{eq:pep-ppa}. For writing~\eqref{eq:pep-ppa} as a semidefinite program, let us introduce base vectors $\bww_k$, $\bvv_k$, and $\bee_k$ in $\mathbb{R}^{2N+4}$ for conveniently selecting entries of $X$, and $\bhh_k$ in $\mathbb{R}^{N+1}$ for selecting entries of $H$, such that
\begin{align*}
    w_k=X\bww_k,\, &v_k=X\bvv_k,\, e_k=X\bee_k,\\
    &h_k=H\bhh_k.
\end{align*}
More precisely, we pick $\bww_\star=\buu_1$, $\bvv_\star=0$, $\bww_0=\buu_2$, $\bvv_k=\buu_{k+3}$ ($k=0,\hdots,N$),  $\bee_k=\buu_{k+N+4}$ ($k=0,\hdots,N$) with $\buu_i$ the unit vector of $\R^{2N+4}$ with $1$ at its $i$th component. For $\bww_k$ ($k=1,\hdots,N$), we use~\eqref{eq:generic-prox} and write
\[   \bww_{k+1} = \bww_k - \sum_{i=1}^{k}\alpha_{k+1,i}\bvv_{i} - \sum_{i=0}^{k}\beta_{k+1,i}\bee_i - \lambda_{k+1}(\bvv_{k+1}+\bee_{k+1}).\]
For function values, we define $\bhh_\star=0$ and $\bhh_k=\buu_{k+1}$ ($k=0,\hdots,N$) with $\buu_i$ now in $\R^{N+1}$. In addition, when the constraints $\text{EQ}_k(S,e)\leq 0$ are Gram-representable, that is, each $\text{EQ}_k(S,e)\leq 0$ can be encoded as $m\in \N^*$ inequalities of the form $\text{EQ}_{k,i}=\mathrm{Tr}(A_{k,i} G)+Ha_{k,i} \leq b_{k,i}$ with $A_{k,i}\in\R^{(2N+4)\times(2N+4)},a_{k,i}\in\R^{N+1},b_{k,i}\in\R$ and $i\in[1,m]$, \eqref{eq:pep-ppa} can finally be reformulated as
\begin{equation}\label{eq:pep-SDP}
\begin{aligned}
\max_{G\succeq 0,\, H}\,\, &H({\bhh_N-\bhh_\star})&\\
\text{s.t. }&\;0\; \geq H(\bhh_i-\bhh_j)+\bvv_i^T G(\bww_j-\bww_i) \quad \forall i,j\in\{\star,0,\hdots,N\}\\
&R^2 \hspace{-0.04cm}\geq (\bww_0-\bww_\star)^T G (\bww_0-\bww_\star)\\
&\;0\; \geq -b_{i,j} + Ha_{i,j} + \mathrm{Tr}(A_{i,j} G)  \quad \forall i\in\{0,\hdots,N\},\,j\in\{1,\hdots,m\},
\end{aligned}
\end{equation}
which is a linear semidefinite program. Feasible points correspond to discrete version of functions $h\in\Fccp$, which can be constructed through convex interpolation mechanisms~\cite{taylor2017smooth}. 

\begin{remark}
The $b_{i,j}$ terms in the inexactness criterion is here to take into account possible absolute (non-homogeneous) error terms (i.e., independent of the iterates).
\end{remark}

We have seen how to solve numerically the performance estimation problem \eqref{eq:pep} using a semidefinite reformulation \eqref{eq:pep-SDP}. The objective of the next section is to present some duality arguments that allows to derive worst-case guarantees from feasible dual points of problem~\eqref{eq:pep-SDP}.

\subsection{Recovering worst-case guarantees from dual solutions}\label{ss:dual}

The worst-case guarantees presented in the sequel were found using dual certificates (i.e., dual feasible points) of problem \eqref{eq:pep-SDP}. In this section, we detail the relationship between such dual feasible points and traditional proofs not relying on semidefinite programming.

Let $\nu = \{\nu_{i,j}\}_{ij}$ be the nonnegative Lagrangian multipliers associated with the convex interpolation constraints and $\mu=\{\mu_{i,j}\}_{ij}$ the ones associated with inexactness constraints. We introduce the quantities
\[ \begin{array}{ccl}
     \tilde{H}(\nu,\mu) &=& \displaystyle\underset{i,j\in\{\star,0,\hdots,N\}}{\sum}\nu_{i,j}\left[(\bhh_i-\bhh_j)\right]
     +\sum_{\substack{i\in\{0,\hdots,N\}\\j\in\{1,\hdots,m\}}}\mu_{i,j}a_{i,j},  \\\\
     \tilde{G}(\nu,\mu) &=& \displaystyle\sum_{i,j\in\{\star,0,\hdots,N\}}\nu_{i,j}\left[(\bww_j-\bww_i)\bvv_i^T \right]+\sum_{\substack{i\in\{0,\hdots,N\}\\j\in\{1,\hdots,m\}}}\mu_{i,j}A_{i,j}, \\\\
     \tilde{B}(\mu) &=& \,\,\,\,\displaystyle\sum_{\substack{i\in\{0,\hdots,N\}\\j\in\{1,\hdots,m\}}}\mu_{i,j}b_{i,j},
\end{array}\]
and the Lagrangian of problem \eqref{eq:pep-SDP} can be expressed as
\begin{equation*}
    \begin{aligned}
    \mathcal{L}(G,H,\nu,\mu,\tau) =& H\left[\bhh_N-\bhh_\star- \tilde{H}(\nu,\mu) \right] + \tau R^2 + \tilde{B}(\mu) -\mathrm{Tr}\left(\left[\tilde{G}(\nu,\mu)+\tau (\bww_0-\bww_\star)
{(\bww_0-\bww_\star)}^T\right]G\right),
    \end{aligned}
\end{equation*}
where $\tau\geq0$ is the multiplier associated with the constraint on distance to optimality of the starting point.

Since the Lagrangian is linear in $G$ and $H$, maximizing with respect to $H$ and $G \succeq 0$ leads to the following dual function
\[\underset{\substack{G\succeq0\\H}}{\max}\;\mathcal{L}(G,H,\nu,\mu,\tau,D) = \left\{\begin{array}{cl}
     \tau R^2+\tilde{B}(\mu) & \text{ if }\bhh_N-\bhh_\star = \tilde{H}(\nu,\mu) \\
     &\text{ and }\tfrac{\tilde{G}(\nu,\mu)+\tilde{G}(\nu,\mu)^T}{2}+\tau (\bww_0-\bww_\star)
{(\bww_0-\bww_\star)}^T \succeq 0 \\
     +\infty & \text{ otherwise}
\end{array} \right.\]
and the corresponding dual problem
\begin{equation}\label{eq:pep-dual}
\begin{aligned}
\min_{\substack{\tau \geq 0,\\\nu \geq 0,\, \mu \geq 0}}\,\, &\tau R^2+\tilde{B}(\mu)&\\
\text{s.t. }&\bhh_N-\bhh_\star = \tilde{H}(\nu,\mu)\\
& \tfrac{\tilde{G}(\nu,\mu)+\tilde{G}(\nu,\mu)^T}{2}+\tau (\bww_0-\bww_\star)
{(\bww_0-\bww_\star)}^T \succeq 0.
\end{aligned}
\end{equation}
Therefore, for any feasible dual point $(\nu,\mu,\tau)$, the following inequality is valid for all $G\succeq0,\,H$
\[ \mathcal{L}(G,H,\nu,\mu,\tau) \leq \tau R^2+\tilde{B}(\mu), \]
which can be rewritten as 
\begin{align*}
    \mathcal{L}(G,H,\nu,\mu,\tau) - \tau R^2 + \tilde{B}(\mu) &= H\left[\bhh_N-\bhh_\star - \tilde{H}(\nu,\mu) \right]-\mathrm{Tr}\left(\left[\tilde{G}(\nu,\mu)+\tau(\bww_0-\bww_\star)
{(\bww_0-\bww_\star)}^T\right]G\right)\\
&\leq \;0.
\end{align*}
Going back to the notations of problem \eqref{eq:pep-ppa} the previous inequality is equivalent to 
\begin{equation}\label{eq:generic-proof}
\begin{aligned}
    h(w_N)-h_\star - \tau\normsq{w_0-w_\star} \leq& \sum_{i,j\in\{\star,0,\hdots,N\}}\nu_{i,j}[h(w_i)-h(w_j)+\inner{v_i}{w_j-w_i}]\\
    &+\sum_{\substack{i\in\{0,\hdots,N\}\\j\in\{1,\hdots,m\}}}\mu_{i,j}\text{EQ}_{i,j} + \tilde{B}(\mu)\\
    \leq& \;\tilde{B}(\mu),
\end{aligned}
\end{equation}
the last inequality comes from the fact that the dual variables are (element-wise) nonnegative, $v_i\in\partial h (w_i)$, and $\text{EQ}_{i,j}\leq 0$. Therefore, we get that 
\[h(w_N)-h_\star \leq \tau\normsq{w_0-w_\star} + \tilde{B}(\mu).\]
Thus, obtaining admissible dual points $\tau$, $\nu$, $\mu$ of problem \eqref{eq:pep-SDP} provides a way of combining interpolation inequalities and inexactness criterion such that \eqref{eq:generic-proof} is valid (examples of proofs relying on this mechanism can be found e.g.,~\cite{de2017worst,lieder_halpern,taylor2019stochastic}).
\begin{remark}
The quantity $\tau R^2+\tilde{B}(\mu)$ is always an upper-bound on $C(N,R)$ when $(\tau,\nu,\mu)$ is a feasible point of \eqref{eq:pep-dual}. Furthermore, under mild conditions for zero duality gap to occur (e.g., when Slater's condition holds for the primal problem), the smallest possible $C(N,R)$ satisfying \eqref{eq:wc_type} is exactly equal to $\tau_\star R^2+\tilde{B}(\mu_\star)$ where $(\tau_\star,\nu_\star,\mu_\star)$ is an optimal solution to~\eqref{eq:pep-dual}.
\end{remark}

\begin{remark}\label{rem:no-abs-error}
When there is no absolute error in the proximal computations (i.e., $b_{i,j}=0$) which corresponds to inequalities $\text{EQ}_k$ that are $1$-homogeneous in function values and $2$-homogeneous in vector variables,
then $\tilde{B}(\mu) = 0$ and the convergence guarantees have the standard form $h(w_N)-h_\star \leq \tau\normsq{w_0-w_\star}$. In addition, we notice that solutions to the dual problem \eqref{eq:pep-dual} are independent of $R$ and the optimal objective value is proportional to $R^2$.
\end{remark}

In the rest of the paper we use this framework to analyze some optimization methods with inexact proximal computations under different inexactness criteria. 

\subsection{Numerical examples} \label{subsec:ppa}
In this section we instantiate various inexact proximal minimization methods and exhibits numerical worst-case guarantees using the framework of Section~\ref{sec:pep}.
\subsubsection{A simple relatively inexact proximal point method}
The \emph{inexact proximal minimization algorithm} with fixed step size presented in Section~\ref{sss:examples_algo} corresponds to updates $w_{k+1} = w_k -\lambda(v_{k+1}+e_{k+1})$, with $v_{k+1}\in \partial h(w_{k+1})$, where we impose a criterion of the form (Primal-dual inaccuracy, take III) that is controlled relatively by the distance between two consecutive iterates. This corresponds to \[\text{EQ}_{k+1} =  \|e_{k+1}\|^2 -\tfrac{\sigma^2}{\lambda}\|w_{k+1}-w_{k}\|^2 \leq 0 \] for a fixed $\sigma \geq 0$. In this setting, problem \eqref{eq:pep-SDP} is of the form

\begin{equation}\label{eq:sdp-ppa}
\begin{aligned}
\max_{G\succeq 0,\, H}\,\, &H({\bhh_N-\bhh_\star})&\\
\text{s.t. }&\;0\; \geq H(\bhh_i-\bhh_j)+\bvv_i^T G(\bww_j-\bww_i) \quad \forall i,j\in\{\star,0,\hdots,N\}\\
&R^2 \hspace{-0.04cm}\geq (\bww_0-\bww_\star)^T G (\bww_0-\bww_\star)\\
&\;0\; \geq \bee_i^TG\bee_i - \tfrac{\sigma^2}{\lambda^2}(\bww_i - \bww_{i-1})^TG(\bww_{i}-\bww_{i-1}) \quad \forall i\in\{1,\hdots,N\},
\end{aligned}
\end{equation}
using notations of Section~\ref{ss:sdp}.

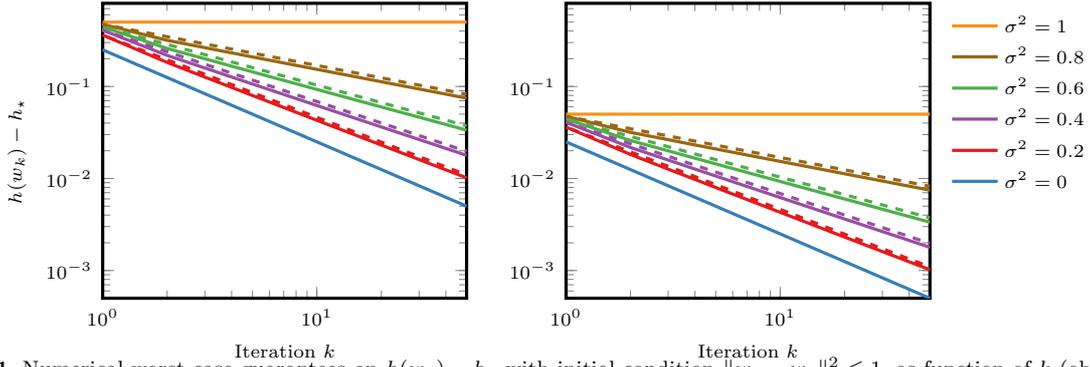
\begin{figure}[!ht]
\centering
\begin{tabular}{cc}
     \begin{tikzpicture}
			\begin{loglogaxis}[legend pos=outer north east, legend style={draw=none},legend cell align={left},plotOptions, ylabel={$h(w_{k})-h_\star$}, ymin=0.0005, ymax=0.8,xmin=1,xmax=50,width=.4\linewidth]
			\addplot [colorP6] table [x=k,y=wc_1]  {figMB/PPA_crit_grad_relat_dist_gamma_1.txt};
			\addplot [colorP5] table [x=k,y=wc_08]  {figMB/PPA_crit_grad_relat_dist_gamma_1.txt};
			\addplot [colorP4] table [x=k,y=wc_06]    {figMB/PPA_crit_grad_relat_dist_gamma_1.txt};
			\addplot [colorP3] table [x=k,y=wc_04]  {figMB/PPA_crit_grad_relat_dist_gamma_1.txt};
			\addplot [colorP2] table [x=k,y=wc_02]  {figMB/PPA_crit_grad_relat_dist_gamma_1.txt};
			\addplot [colorP1] table [x=k,y=wc_0] {figMB/PPA_crit_grad_relat_dist_gamma_1.txt};
			
			\addplot [color=colorP2, dashed, domain=0.01:50, samples=20] {(1+sqrt(0.2))/4/(x)^(sqrt(1-0.2))};
			\addplot [color=colorP3, dashed, domain=0.01:50, samples=20] {(1+sqrt(0.4))/4/(x)^(sqrt(1-0.4))};
			\addplot [color=colorP4, dashed, domain=0.01:50, samples=20] {(1+sqrt(0.6))/4/(x)^(sqrt(1-0.6))};
			\addplot [color=colorP5, dashed, domain=0.01:50, samples=20] {(1+sqrt(0.8))/4/(x)^(sqrt(1-0.8))};
			\end{loglogaxis}
		\end{tikzpicture} &
		\begin{tikzpicture}
			\begin{loglogaxis}[legend pos=outer north east, legend style={draw=none},legend cell align={left},plotOptions, ymin=0.0005, ymax=0.8,xmin=1,xmax=50,width=.4\linewidth]
			\addplot [colorP6] table [x=k,y=wc_1]  {figMB/PPA_crit_grad_relat_dist_gamma_10.txt};
			\addplot [colorP5] table [x=k,y=wc_08]  {figMB/PPA_crit_grad_relat_dist_gamma_10.txt};
			\addplot [colorP4] table [x=k,y=wc_06]    {figMB/PPA_crit_grad_relat_dist_gamma_10.txt};
			\addplot [colorP3] table [x=k,y=wc_04]  {figMB/PPA_crit_grad_relat_dist_gamma_10.txt};
			\addplot [colorP2] table [x=k,y=wc_02]  {figMB/PPA_crit_grad_relat_dist_gamma_10.txt};
			\addplot [colorP1] table [x=k,y=wc_0] {figMB/PPA_crit_grad_relat_dist_gamma_10.txt};
			
			\addlegendentry{$\sigma^2=1$}
			\addlegendentry{$\sigma^2=0.8$}
			\addlegendentry{$\sigma^2=0.6$}
			\addlegendentry{$\sigma^2=0.4$}
			\addlegendentry{$\sigma^2=0.2$}
			\addlegendentry{$\sigma^2=0$}
			
			\addplot [color=colorP2, dashed, domain=0.01:50, samples=20] {(1+sqrt(0.2))/4/(x)^(sqrt(1-0.2))/10};
			\addplot [color=colorP3, dashed, domain=0.01:50, samples=20] {(1+sqrt(0.4))/4/(x)^(sqrt(1-0.4))/10};
			\addplot [color=colorP4, dashed, domain=0.01:50, samples=20] {(1+sqrt(0.6))/4/(x)^(sqrt(1-0.6))/10};
			\addplot [color=colorP5, dashed, domain=0.01:50, samples=20] {(1+sqrt(0.8))/4/(x)^(sqrt(1-0.8))/10};
			\end{loglogaxis}
		\end{tikzpicture}\vspace{-.5cm}
\end{tabular}
		\caption{Numerical worst-case guarantees on $h(w_{k})-h_\star$ with initial condition $\normsq{w_0-w_\star}\leq 1$, as function of $k$ (obtained by solving semidefinite programs \eqref{eq:sdp-ppa}) for the relatively inexact proximal point algorithm of Section \ref{subsec:ppa}, with parameters $\lambda = 1$ (left), and $\lambda =10$ (right). The dashed lines are empirical upper bounds of the form $({1+\sigma})/({4\lambda k^{\sqrt{1-\sigma^2}}})$ which we plotted for reference. The semidefinite programs were solved through~\cite{Article:Yalmip} and~\cite{Article:Mosek}.} \label{fig:relative_PPA}
\end{figure}

One can now solve~\eqref{eq:sdp-ppa} numerically, for different values of $\sigma$, $\lambda$ and $R$, using standard semidefinite solvers (see e.g.; \cite{Article:Mosek,Article:Sedumi}). The corresponding numerical worst-case bounds are provided in Figure~\ref{fig:relative_PPA} for different parameter values. Based on numerical experiments, we conjecture the expression $R^2({1+{\sigma}})/({4\lambda N^{\sqrt{1-\sigma^2}}})$ to be a valid $C(N,R)$. For this example, we do not have a proof for this bound, as the algebra involved in obtaining an analytical form of a dual feasible point (as described in Section~\ref{ss:dual}) turned out to be quite complicated in our trials on this simple method.


This example illustrates how we can use the performance estimation approach to compute worst-case bounds numerically, even when rigorous analytical proofs seem out of reach.

\subsubsection{Inexact accelerated proximal point algorithms \emph{IAPPA}}
As detailed in Appendix~\ref{app:methods}, \emph{IAPPA1} and \emph{IAPPA2} from \cite[Section 5]{salzo2012inexact} fit into the formalism of Section~\ref{sec:pep}. In particular, one can apply Section~\ref{ss:sdp} to compute numerical worst-case guarantees, as provided in Figure~\ref{fig:IAPPA}. 

\begin{figure}[!ht]
\centering
\begin{tabular}{cc}
\begin{tikzpicture}
			\begin{loglogaxis}[legend pos=outer north east, legend style={draw=none},legend cell align={left},plotOptions, ymin=0.0002, ymax=5,xmin=1,xmax=120,width=.4\linewidth, ylabel={$h(x_{k})-h_\star$}]
			
			\addplot [colorP6] table [x=N,y=wc,col sep=comma] {figMB/IAPPA1_absolute_-1_PD_gapI-150-.txt};
			\addplot [colorP5] table [x=N,y=wc,col sep=comma] {figMB/IAPPA1_absolute_-1.5_PD_gapI-150-.txt};
			\addplot [colorP4] table [x=N,y=wc,col sep=comma] {figMB/IAPPA1_absolute_-2_PD_gapI-150-.txt};
			\addplot [colorP3] table [x=N,y=wc,col sep=comma] {figMB/IAPPA1_absolute_-3_PD_gapI-150-.txt};
			\addplot [colorP2] table [x=N,y=wc,col sep=comma] {figMB/IAPPA1_absolute_-4_PD_gapI-150-.txt};
			\addplot [colorP1] table [x=N,y=wc,col sep=comma] {figMB/IAPPA1_absolute_no-error-150.txt};

            \addplot[colorP1,dashed,domain=20:150,samples=10] {2.4/x^2};
			
			\end{loglogaxis}
\end{tikzpicture}
    &
		\begin{tikzpicture}
			\begin{loglogaxis}[legend pos=outer north east, legend style={draw=none},legend cell align={left},plotOptions, ymin=0.0002, ymax=5,xmin=1,xmax=120,width=.4\linewidth]
		
			\addplot [colorP6] table [x=N,y=wc,col sep=comma] {figMB/IAPPA1_absolute_-1_PD_gapIII-150.txt};
			\addplot [colorP5] table [x=N,y=wc,col sep=comma] {figMB/IAPPA1_absolute_-1.5_PD_gapIII-150.txt};
			\addplot [colorP4] table [x=N,y=wc,col sep=comma] {figMB/IAPPA1_absolute_-2_PD_gapIII-150.txt};
			\addplot [colorP3] table [x=N,y=wc,col sep=comma] {figMB/IAPPA1_absolute_-3_PD_gapIII-150.txt};
			\addplot [colorP2] table [x=N,y=wc,col sep=comma] {figMB/IAPPA1_absolute_-4_PD_gapIII-150.txt};
			\addplot [colorP1] table [x=N,y=wc,col sep=comma] {figMB/IAPPA1_absolute_no-error-150.txt};
			\addlegendentry{$\varepsilon_k = k^{-1}$}
			\addlegendentry{$\varepsilon_k = k^{-\tfrac{3}{2}}$}
			\addlegendentry{$\varepsilon_k = k^{-2}$}
			\addlegendentry{$\varepsilon_k = k^{-3}$}
			\addlegendentry{$\varepsilon_k = k^{-4}$}
			\addlegendentry{$\varepsilon_k = 0$}
			
			\addplot[colorP1,dashed,domain=20:150,samples=10] {2.4/x^2};
			\addlegendentry{$O(k^{-2})$}
			\end{loglogaxis}
		\end{tikzpicture}
\end{tabular}
		\caption{Numerical worst-case guarantees on $h(x_{k})-h_\star$ with initial condition $\normsq{w_0-w_\star}\leq 1$, as function of $k$ for \emph{IAPPA1} (left) and \emph{IAPPA2} (right), with constant step size equal to $1$ and $(\varepsilon_k)_k$ the sequence of parameters controlling the primal-dual gap values. We observe that for \emph{IAPPA1} (left), the cases $\varepsilon_k = k^{
-4}$ (red) and  $\varepsilon_k = k^{-3}$ (purple) seems to decrease as  $O(k^{-1})$ as stated in \cite[Theorem 4]{salzo2012inexact}. For \emph{IAPPA2} (right), \cite[Theorem 6]{salzo2012inexact} states that $\varepsilon_k = k^{
-4}$ (red) and $\varepsilon_k = k^{-3}$ (purple) curves of Figure~\ref{fig:IAPPA} (right) should exhibit a convergence in $O(k^{-2})$, as observed. More iterations might be needed to observe the same phenomenon for the $\varepsilon_k=k^{-2}$ (green).} \label{fig:IAPPA}
\end{figure}
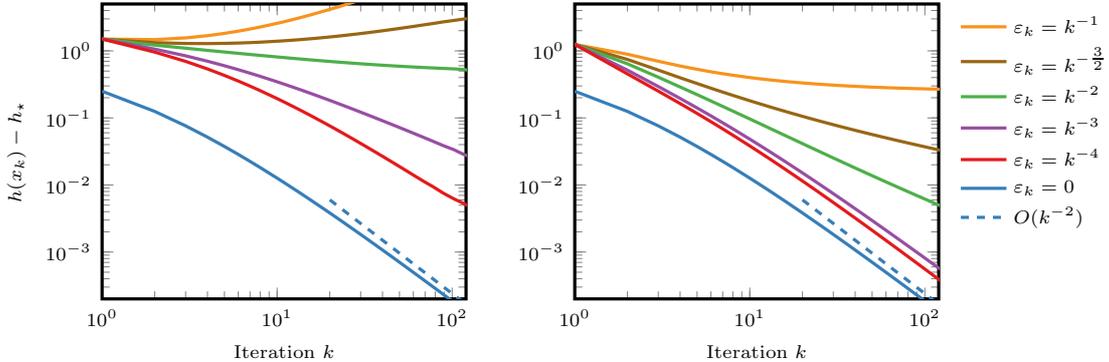
Regarding the numerical experiments, note that it might be delicate to deduce asymptotic convergence convergence rates by looking only at about a hundred of iterations. This is the limiting part of this approach: the number of constraints in the semidefinite problems defined in Section~\ref{ss:sdp} grows with the square of the number of iterations, which limits our capabilities of solving the corresponding problem. However, we can still make some observations, and sometimes deduce proofs (see Section~\ref{ss:dual}).

Let us compare numerical worst-case guarantees in Figure~\ref{fig:IAPPA} with convergence theorems  \cite[Theorem 4, Theorem 6]{salzo2012inexact} for \emph{IAPPA1} and \emph{IAPPA2}. First note that \cite[Theorem 4]{salzo2012inexact} states that primal gap in \emph{IAPPA1} converges to $0$ as soon as $\varepsilon_k = O(k^{-q})$ with $q>\tfrac{3}{2}$, which is compatible with numerical experiments in Figure~\ref{fig:IAPPA} (left). Reciprocally, it does seem that $q\leq \tfrac{3}{2}$ the worst-case guarantee does not converge to $0$, apparently tightening \cite[Theorem 4]{salzo2012inexact}. Similar observations hold for algorithm \emph{IAPPA2} (which involves a stricter inexactness requirements) with convergence of the primal gap for $q > 1/2$.

\subsubsection{Accelerated hybrid proximal Extragradient method (A-HPE)}
As detailed in Appendix~\ref{app:methods}, the \emph{A-HPE} method from \cite[Section 3]{monteiro2013accelerated} also fits into the formalism of Section~\ref{sec:pep}. In particular, one can apply Section~\ref{ss:sdp} to compute numerical worst-case guarantees that we provide in Figure~\ref{fig:AHPE}. 
\begin{figure}[!ht]
\centering
\begin{tabular}{cc}
\begin{tikzpicture}
			\begin{loglogaxis}[legend pos=outer north east, legend style={draw=none},legend cell align={left},plotOptions, ymin=0.001, ymax=0.5,xmin=1,xmax=20,width=.4\linewidth,ylabel={$h(y_k)-h_\star$}]
			
			\addplot [color=black, dashed,domain=1:10, samples=2]  {10};
			
			\addplot [colorP5] table [x=N,y=wc1,col sep=comma] {figMB/A_HPE_PD_gapI.txt};
			\addplot [colorP4] table [x=N,y=wc075,col sep=comma] {figMB/A_HPE_PD_gapI.txt};
			\addplot [colorP3] table [x=N,y=wc05,col sep=comma] {figMB/A_HPE_PD_gapI.txt};
			\addplot [colorP2] table [x=N,y=wc025,col sep=comma] {figMB/A_HPE_PD_gapI.txt};
			\addplot [colorP1] table [x=N,y=wc0,col sep=comma] {figMB/A_HPE_PD_gapI.txt};

			\addplot [color=black, dashed] table [x=N,y=Ak,col sep=comma] {figMB/A_HPE_PD_gapI.txt};
            
			\end{loglogaxis}
\end{tikzpicture}     &
\begin{tikzpicture}
			\begin{loglogaxis}[legend pos=outer north east, legend style={draw=none},legend cell align={left},plotOptions, ymin=0.001, ymax=0.5,xmin=1,xmax=20,width=.4\linewidth]
			
			\addplot [color=black, dashed,domain=1:10, samples=2]  {10};
            \addlegendentry{$\tfrac{1}{2A_k}$}
			
			\addplot [colorP5] table [x=N,y=wc1,col sep=comma] {figMB/A_HPE_PD_gapI_stepsize_10.txt};
			\addplot [colorP4] table [x=N,y=wc075,col sep=comma] {figMB/A_HPE_PD_gapI_stepsize_10.txt};
			\addplot [colorP3] table [x=N,y=wc05,col sep=comma] {figMB/A_HPE_PD_gapI_stepsize_10.txt};
			\addplot [colorP2] table [x=N,y=wc025,col sep=comma] {figMB/A_HPE_PD_gapI_stepsize_10.txt};
			\addplot [colorP1] table [x=N,y=wc0,col sep=comma] {figMB/A_HPE_PD_gapI_stepsize_10.txt};
			
            \addlegendentry{$\sigma = 1$}
			\addlegendentry{$\sigma = 0.75$}
			\addlegendentry{$\sigma = 0.5$}
			\addlegendentry{$\sigma = 0.25$}
			\addlegendentry{$\sigma = 0$}
			
			\addplot [color=black, dashed] table [x=N,y=Ak,col sep=comma] {figMB/A_HPE_PD_gapI_stepsize_10.txt};
            
			\end{loglogaxis}
\end{tikzpicture}
\end{tabular}

		\caption{Numerical worst-case guarantees on $h(y_k)-h_*$ with initial condition $\normsq{w_0-w_\star}\leq 1$, as function of $k$ for the \emph{A-HPE} method with constant step size $\lambda=1$ (left), and $\lambda=10$ (right). The dashed curve corresponds to a theoretical upper bound on the primal gap from \cite[Theorem 3.6]{monteiro2013accelerated}.} \label{fig:AHPE}
\end{figure}
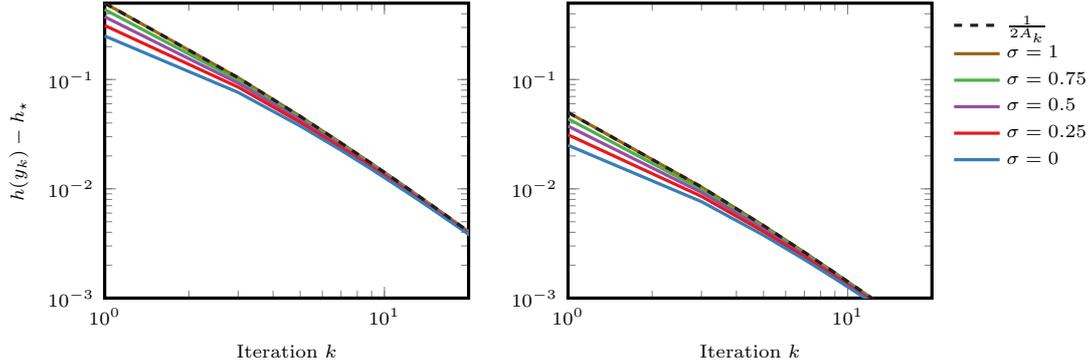

The numerical bounds on $h(y_k)-h_*$ that we obtain in Figure~\ref{fig:AHPE} for $\sigma = 1$ seems to match exactly the analytical bound $\tfrac{\normsq{w_0-w_\star}}{2A_k}$ provided in \cite[Theorem 3.6]{monteiro2013accelerated}. We further observe that numerical worst-case guarantees for all $\sigma\in[0,1]$ tend to match with this analytical bound when the number of iterations gets larger.

In the next section we describe an optimized relatively inexact proximal point method with worst-case behaviour derived from a dual feasible point, as previously described in Section~\ref{ss:dual}.
\section{An optimized relatively inexact proximal point algorithm}\label{sec:orip}
In this section we use the framework detailed in the Section~\ref{sec:pep} for designing an inexact proximal minimization algorithm with optimized worst-case performances. Similar to \eqref{eq:generic-prox}, provided sequences of step sizes $\{\lambda_i\}_i$ and parameters $\{\alpha_{i,j}\}_{ij}$, $\{\beta_{i,j}\}_{ij}$, we consider iterates of the form 
\begin{equation}\label{eq:generic-orip}
    \begin{aligned}
    x_{k+1} = x_k - \sum_{i=1}^k\alpha_{k+1,i}g_i - \sum_{i=1}^k\beta_{k+1,i}e_i - \lambda_{k+1}(g_{k+1}+e_{k+1}),
    \end{aligned}
\end{equation}
and impose an inexactness criterion of the form 
\[\PDg_{\lambda_{k}h}(x_{k},g_k;x_{k}+\lambda_k(g_k+e_k))\leq \tfrac{\sigma^2}{2}\normsq{\lambda_k(g_k+e_k)} , \quad k \geq 1. \]
This class of methods actually fits into \eqref{eq:generic-prox} and \eqref{eq:generic_inexact}, as shown in the next section.

Note that as mentioned in Remark~\ref{rem:no-abs-error}, in the absence of non-homogeneous error terms in the inexactness criteria (which is the case here) and given methods parameters, provable worst-case guarantees derived from dual certificates are independent of the bound on the initial distance to optimality $R$. Therefore, we fix $R=1$ in the performance estimation problems studied in this section for simplicity.

In order to find parameters $\{\alpha_{i,j}\}_{ij}$ and $\{\beta_{i,j}\}_{ij}$ that provide the smallest possible worst-case guarantees on $h(x_N)-h_\star$ after $N\in\N^*$ iterations, we define
\begin{equation}\label{eq:W}
\begin{aligned}
 W\left(\{\alpha_{i,j}\}_{ij},\{\beta_{i,j}\}_{ij}\right) := \max_{\substack{d,h\\x_\star,x_0,\hdots,x_{N}\in\Rd\\g_0,\hdots,g_{N}\in\Rd\\e_0,\hdots,e_{N}\in\Rd}}& h(x_{N})-h(x_\star)&\\
\text{s.t. }&h\in\Fccp(\Rd), \quad x_\star\in\argmin_x h(x)\\
& \normsq{x_0-x_\star} \leq 1\\
&x_1,\hdots,x_{N} \text{ satisfying } \eqref{eq:generic-orip}\\
&\PDg_{\lambda_{k}h}(x_{k},g_k;x_{k}+\lambda_k(g_k+e_k))\\
&\leq \tfrac{\sigma^2}{2}\normsq{\lambda_k(g_k+e_k)} \quad k=1,\hdots,{N},
\end{aligned}
\end{equation}
and wish to solve the following problem 
\begin{equation}\label{eq:orippep}
\begin{aligned}
 \underset{\{\alpha_{i,j}\}_{ij},\{\beta_{i,j}\}_{ij}}{\argmin}\;W\left(\{\alpha_{i,j}\}_{ij},\{\beta_{i,j}\}_{ij}\right).
\end{aligned}
\end{equation}


The rest of the section is organized as follow. First, we reformulate the method \eqref{eq:generic-orip} and problem \eqref{eq:W} for fitting into the setting and notations of Section~\ref{sec:pep}. Then, since solving \eqref{eq:orippep} exactly is out of reach in general, we detail a procedure to obtain feasible points (i.e., methods parameters) with optimized objective value. Finally, we present the method obtained from this choice of parameters together with its worst-case analysis.

\subsection{Reformulation as fixed-step inexact proximal methods}
The difference between \eqref{eq:generic-orip} and \eqref{eq:generic-prox} lies in the fact that we do not enforce $g_i \in \partial h(x_i)$ in the first model.
In order to cast \eqref{eq:generic-orip} into \eqref{eq:generic-prox}, we can define iterates as
\begin{equation}\label{eq:generic-prox-optim}
\left\{\begin{array}{ccl}
     w_{2k+1}&=& w_{2k}-e_{2k}  \\
     w_{2k+2}&=& w_{2k} - \sum_{i=1}^{k}\alpha_{k+1,i}v_{2i-1} - \sum_{i=1}^{k}\beta_{k+1}e_{2i-1} -\lambda_{k+1}(v_{2k+1}+e_{2k+1}), 
\end{array}\right.    
\end{equation}
with $v_{i} \in \partial h (w_i)$, which fits into \eqref{eq:generic-prox}.

The inexactness requirements are then, $\text{EQ}_0 =0$, $\text{EQ}_{2k+1}=0$ and 
\begin{align*}
    \text{EQ}_{2k+2} =& \PDg_{\lambda_{k+1}h}(w_{2k+2},v_{2k+1},w_{2k+2}+\lambda_{k+1}(v_{2k+1}+e_{k+1})) - \tfrac{\sigma^2}{2}\normsq{\lambda_{k+1}(v_{2k+1}+e_{2k+1})}.
\end{align*}
Since $v_{2k+1}\in \partial h (w_{2k+1})$, we can write \[h^*(v_{2k+1}) = \inner{v_{2k+1}}{w_{2k+1}} - h(w_{2k+1}),\] in the primal-dual gap and thus
\begin{align*}
\text{EQ}_{2k+2}  =& \tfrac{\lambda_{k+1}}{2}\normsq{ e_{2k+1}} - \tfrac{\lambda_{k+1}\sigma^2}{2}\normsq{v_{2k+1}+e_{2k+1}}+h(w_{2k+2})-h(w_{2k+1}) - \inner{v_{2k+1}}{w_{2k+2}-w_{2k+1}}.    
\end{align*}
which is Gram-representable. 
Finally, we can identify iterates $\{w_{2k}\}_k$ with the $\{x_k\}_k$ from \eqref{eq:generic-orip} and we have
\begin{equation}\label{eq:design-orip}
\begin{aligned}
W\left(\{\alpha_{i,j}\}_{ij},\{\beta_{i,j}\}_{ij}\right) =\max_{\substack{d,h\\w_\star,w_0,\hdots,w_{2N}\in\Rd\\v_0,\hdots,v_{2N}\in\Rd\\e_0,\hdots,e_{2N}\in\Rd}}& h(w_{2N})-h(w_\star)&\\
\text{s.t. }&h\in\Fccp(\Rd), \quad w_\star\in\argmin_x h(x)\\
& \normsq{w_0-w_\star} \leq 1\\
&w_1,\hdots,w_{2N} \text{ satisfying } \eqref{eq:generic-prox-optim}\\
&\text{EQ}_k \leq 0 \quad k=0,\hdots,{N}.
\end{aligned}
\end{equation}

In the following we first give a high level overview of how we can use a relaxation of \eqref{eq:design-orip} inside the minimization problem \eqref{eq:orippep} to get a feasible point with optimized worst-case bound, and then present the algorithm obtained with this choice of optimized parameters together with its sharp convergence guarantees (sharp in the sense that given $\{\lambda_k\}_k$ and $N$ we can find a function for which the worst-case guarantee is attained exactly).
\subsection{Obtaining optimized parameters}\label{ss:optim}
Problem \eqref{eq:orippep} can be formulated as a linear minimization problem under a bilinear matrix inequality, which is NP-hard in general (see e.g., \cite{toker1995np}). Thus, we approximate it by using a technique similar to that of~\cite{drori2014performance,kim2016optimized}, which consists in four steps.
\begin{enumerate}[label=(\roman*)]
    \item  Find a suitable relaxation of the inner maximization problem \eqref{eq:design-orip} i.e., only keep a subset of the interpolation constraints. This relaxation is chosen by a numerical trial and error procedure.
    \item Dualize the relaxed semidefinite formulation of the inner maximization problem to obtain a two-level minimization problem.
    \item Use a change of variable similar to that in~\cite[Section 5]{drori2014performance} to remove nonlinear terms in the bilinear semidefinite problem obtained at the previous step.
    \item Retrieve a feasible point of \eqref{eq:orippep} from the solution of the problem obtained in step (iii).
\end{enumerate}
     The final choice for the relaxation of \eqref{eq:design-orip} consisted in using only the following interpolation inequalities:
\begin{itemize}
\item convexity inequality between $w_{2k}$ and $w_{2k+1}$
\[h(w_{2k})\geq h(w_{2k+1})+\inner{v_{2k+1}}{w_{2k}-w_{2k+1}}, \]
\item convexity inequality between $w_\star$ and $w_{2k+1}$ 
\[h(w_\star)\geq h(w_{2k+1})+\inner{v_{2k+1}}{w_\star-w_{2k+1}}, \]
\item convexity inequality between $w_{2k+2}$ and $w_{2k+1}$ 
\[h(w_{2k+2})\geq h(w_{2k+1})+\inner{v_{2k+1}}{w_{2k+2}-w_{2k+1}}, \]
\end{itemize}
along with inexactness conditions $\text{EQ}_k$. Those are exactly the inequalities used in the proof in next section. 

More precisely, step (i) consisted in replacing $W\left(\{\alpha_{i,j}\}_{ij},\{\beta_{i,j}\}_{ij}\right)$ in problem \eqref{eq:orippep} by a relaxed version $U\left(\{\alpha_{i,j}\}_{ij},\{\beta_{i,j}\}_{ij}\right)$ defined in its semidefinite form as follow 
\begin{equation}\label{eq:sdp-orip}
\begin{aligned}
U\left(\{\alpha_{i,j}\}_{ij}\right.&\left.,\{\beta_{i,j}\}_{ij}\right):=\\
\max_{G\succeq 0,\, H}\,\, &H({\bhh_{2N}-\bhh_\star})&\\
\text{s.t. }&0 \geq H(\bhh_{2i+1}-\bhh_{2i})+\bvv_{2i+1}^T G(\bww_{2i}-\bww_{2i+1}) \quad \forall i\in\{1,\hdots,N-1\}\\
&0 \geq H(\bhh_{2i-1}-\bhh_\star)+\bvv_{2i-1}^T G(\bww_\star-\bww_{2i-1}) \quad \forall i\in\{1,\hdots,N\}\\
&0 \geq H(\bhh_{2i}-\bhh_{2i-1})+\bvv_{2i-1}^T G(\bww_{2i}-\bww_{2i-1}) \quad \forall i\in\{1,\hdots,N\}\\
&1 \geq (\bww_0-\bww_\star)^T G (\bww_0-\bww_\star)\\
&0 \geq \tfrac{\lambda_{i}}{2}\bee_{2i-1}^TG\bee_{2i-1} + H(\bhh_{2i}-\bhh_{2i-1}) - \bvv_{2i-1}^TG(\bww_{2i}-\bww_{2i-1})\\
&\;\;\;\;\;  - \tfrac{\sigma^2\lambda_{i}}{2}(\bee_{2i-1}+\bvv_{2i-1})^TG(\bee_{2i-1}+\bvv_{2i-1}) \quad \forall i\in\{1,\hdots,N\},
\end{aligned}
\end{equation}

Then, step (ii) consisted in dualizing the maximization problem as seen in Section~\ref{ss:dual}. From there, we search for parameters $\{\alpha_{i,j}\}_{ij},\{\beta_{i,j}\}_{ij}$ that minimize the optimal value of the dual of \eqref{eq:sdp-orip}. This is a minimization problem in $\{\alpha_{i,j}\}_{ij},\{\beta_{i,j}\}_{ij}$ and in the dual variables of \eqref{eq:sdp-orip}, that contains bilinear terms. 

In step (iii), the bilinear terms in the minimization problem of step (ii) are replaced by new variables, producing a linear semidefinite program that can be solved efficiently.

Finally, in the last step, we retrieve parameters $\{\tilde{\alpha}_{i,j}\}_{ij}$ and $\{\tilde{\beta}_{i,j}\}_{ij}$ from the solutions of the linear semidefinite program of step (iii). Note that the relaxation step (i) is chosen so that step (iv) is achievable.

In the following, we describe the algorithm obtained from the choice of parameters $\{\tilde{\alpha}_{i,j}\}_{ij}$, $\{\tilde{\beta}_{i,j}\}_{ij}$.
\subsection{Algorithm and convergence guarantees}

\begin{oframed}
	\textbf{Optimized relatively inexact proximal point algorithm (ORI-PPA)}
	\begin{itemize}
		\item[] Input: $h\in\Fccp(\Rd)$, $x_0\in\mathbb{R}^d$, $\sigma \in[0,1]$. 
		\item[] Initialization: $z_0=x_0$, $A_0=0$. 
		\item[] For $k=0,1,\hdots$:
		\begin{equation}\label{eq:orip}\tag{ORI-PPA}
			\begin{aligned}
			\text{Choose } &\lambda_{k+1} \geq 0 \\
			A_{k+1}&=A_k + \tfrac{\lambda_{k+1}+\sqrt{4\lambda_{k+1}A_{k}+\lambda_{k+1}^2}}{2}\\
			y_{k}&=x_k+\tfrac{\lambda_{k+1}}{A_{k+1}-A_k}(z_k-x_{k})\\
			\left[\text{Obtain }\right.(x_{k+1},& \;g_{k+1})\approx\left(\prox_{\lambda_{k+1} h}(y_{k}),\prox_{h^*/\lambda_{k+1}}(\tfrac{y_{k}}{\lambda_{k+1}})\right) \\ 
			\text{which satisfies} &\left.\text{ $\PDg_{\lambda_{k+1} h}(x_{k+1},g_{k+1};y_k) \leq \tfrac{\sigma^2}{2} \|x_{k+1}-y_k\|^2$}\right]\\
			{z}_{k+1}&=z_{k}-\tfrac{2(A_{k+1}-A_k)}{1+\sigma}g_{k+1}\\
			\end{aligned}
			\end{equation}
			\item[] Output: $x_{k+1}$
		\end{itemize}  
	\end{oframed}

Perhaps luckily, it turns out that the parameters $\{\tilde{\alpha}_{i,j}\}_{ij}$ and $\{\tilde{\beta}_{i,j}\}_{ij}$ obtained from the four step procedure of Section~\ref{ss:optim} follow recursive equations allowing to rewrite iterations \eqref{eq:generic-prox} under a more compact form as presented in Algorithm~\eqref{eq:orip} above. As mentionned earlier the iterates $\{x_k\}_k,\{g_{k+1}\}_k$ corresponds to the $\{w_{2k}\}_k,\{v_{2k+1}\}_k$ from \eqref{eq:generic-prox-optim} using$\{\tilde{\alpha}_{i,j}\}_{ij}$ and $\{\tilde{\beta}_{i,j}\}_{ij}$.

The Algorithm~\eqref{eq:orip} is actually almost the same as the A-HPE algorithm from \cite{monteiro2013accelerated} (in particular definitions of sequences $\{y_k\}_k$, $\{z_k\}_k$ are the same when $\sigma = 1$). The main differences reside in the inexactness criterion, as we prefer to use primal-dual formulation rather than using $\varepsilon$-subgradients, and in the fact that \eqref{eq:orip} uses explicitly the inexactness level $\sigma$ in its step sizes. This last difference allows to improve the worst-case guarantee by a constant factor $\tfrac{1+\sigma}{2}\leq 1$ compared to \cite[Theorem 3.6]{monteiro2013accelerated}. 

Perhaps surprisingly, this method reduces to that of G\"uler~\cite[Section 6]{guler1992new} when using exact proximal operations ($\sigma = 0$) and constant step size, although the current method was obtained by crude numerical optimization of its parameters (see Appendix~\ref{app:guler-eq} for details).

Solving numerically the dual of \eqref{eq:sdp-orip} allows to obtain rather simple analytical form for the optimal dual variables. We use these multipliers as in Section~\ref{ss:dual}, to prove the following theorem.

\begin{theorem} \label{thm:orip2}Let $h\in\Fccp$, a sequence of step sizes $\{\lambda_k\}_k$ with $\lambda_k > 0$, and $\sigma\in[0,1]$. For any starting point $x_0\in\Rd$, $N\geq 1$, the iterates of \eqref{eq:orip} satisfy
\[h(x_N)-h(x_\star)\leq \tfrac{(1+\sigma)\normsq{x_0-x_\star}}{4 A_N}, \]
with $x_\star \in \argmin_x h(x)$. Furthermore, this bound is tight: for all $\{\lambda_k\}_k$ with $\lambda_k>0$, $\sigma\in[0,1]$, $d\in\mathbb{N}$, $x_0\in\mathbb{R}^d$, and $N\in\mathbb{N}$, there exists $h\in\Fccp(\Rd)$ such that this bound is achieved with equality.
\end{theorem}
\begin{proof}
For the sake of clarity, we present the proof using notations of \eqref{eq:orip}, although the proof was found via the SDP formulation \eqref{eq:sdp-orip}.

We start with the case $\sigma \in (0,1]$. The case $\sigma=0$ is considered afterward as it requires a slightly different treatment.

In the following we denote by $u_{k+1}$ a point satisfying $u_{k+1}\in \partial h^*(g_{k+1})$ or equivalently $g_{k+1}\in \partial h(u_{k+1})$. These $u_{k+1}$ can be identified with the $w_{2k+1}$ in \eqref{eq:generic-prox-optim}.

Consider the following inequalities with their corresponding weights :
\begin{itemize}
\item convexity between $x_k$ and $u_{k+1}$ with weight $\nu_{k,k+1}=\tfrac{A_k}{1+\sigma}$\\
(for $k=1,\hdots,N-1$)
\[h(x_k)\geq h(u_{k+1})+\inner{g_{k+1}}{x_k-u_{k+1}}, \]
\item convexity between $x_\star$ and $u_k$ with weight $\nu_{\star,k}=\tfrac{A_k-A_{k-1}}{1+\sigma}$ \\(for $k=1,\hdots,N$)
\[h(x_\star)\geq h(u_{k})+\inner{g_k}{x_\star-u_{k}}, \]
\item convexity between $x_k$ and $u_k$ with weight 
$\nu_{k,k}=\tfrac{A_k(1-\sigma)}{\sigma(1+\sigma)}$ \\(for $k=1,\hdots,N$)
\[h(x_k)\geq h(u_{k})+\inner{g_k}{x_k-u_{k}}, \]
\item approximation requirement on $x_k$ with weight $\nu_k=\tfrac{A_k}{\sigma(1+\sigma)}$ \\(for $k=1,\hdots,N$)
\[ \tfrac{ \sigma^2}{2\lambda_k}\normsq{x_k-y_{k-1}} \geq \tfrac{1}{2\lambda_k}\normsq{x_k-y_{k-1}+\lambda_kg_k}+ h(x_k)-h(u_k)-\inner{g_k}{x_k-u_k} .\]
\end{itemize}
By linearly combining the previous inequalities, with their corresponding weights (which are nonnegative), we arrive to the following valid inequality:
		\begin{equation*}
		\begin{aligned}
		\sum_{k=1}^{N-1} &\nu_{k,k+1} h(x_k)+\sum_{k=1}^{N} \nu_{\star,k} h(x_\star)+\sum_{k=1}^N \nu_{k,k}h(x_k)+\sum_{k=1}^N \nu_k \tfrac{\sigma^2}{2\lambda_k}\norm{x_k-y_{k-1}}\\
		\geq&\sum_{k=1}^{N-1}\nu_{k,k+1}[h(u_{k+1}) +\inner{g_{k+1}}{x_k-u_{k+1}}]\\&+\sum_{k=1}^N\nu_{\star,k}[h(u_k)+\inner{g_k}{x_\star-u_{k}}]+\sum_{k=1}^{N}\nu_{k,k}[h(u_{k}) +\inner{g_{k}}{x_k-u_{k}}]\\
		&+\sum_{k=1}^N\nu_k [\tfrac{1}{2\lambda_k}\normsq{x_k-y_{k-1}+\lambda_kg_k}+ h(x_k)-h(u_k)-\inner{g_k}{x_k-u_k}].
		\end{aligned}
		\end{equation*}
		Substituting $x_{k}$ by its expression in~\eqref{eq:eq_orip}, a reasonable amount of work (see Appendix~\ref{sec:missingparts_orip}) allows reformulating this inequality exactly as
		\begin{equation*}
		\begin{aligned}
		\tfrac{A_N}{1+\sigma} (h(x_N)-h_\star) \leq& \tfrac{1}{4}\normsq{x_0-x_\star} - \tfrac14\normsq{x_\star-x_0+\tfrac{2}{1+\sigma}\sum_{i=1}^N(A_i-A_{i-1}) g_i} \\
		&- \tfrac{1-\sigma}{2\sigma}\sum_{i=1}^NA_i\lambda_i\normsq{\tfrac{y_{i-1}-\lambda_ig_i - x_i }{\lambda_i}+ \tfrac{\sigma}{1+\sigma}g_i}.
		\end{aligned}
		\end{equation*}
		Since the last two terms in the right hand side are nonpositive, we deduce that
		\[ \tfrac{A_N}{1+\sigma} (h(x_N)-h_\star)  \leq \tfrac14\normsq{x_0-x_\star}.\]
		
		For the case $\sigma = 0$ \cite[Theorem 6.1]{guler1992new} provides a proof when using constant step sizes. Here, we follow the same pattern as before for allowing variable step sizes. We consider the following inequalities
		\begin{itemize}
        \item convexity between $x_k$ and $x_{k+1}$ with weight $\nu_{k,k+1}=A_k$\\
        (for $k=0,\hdots,N-1$)
        \[h(x_k)\geq h(x_{k+1})+\inner{g_{k+1}}{x_k-x_{k+1}}, \]
        \item convexity between $x_\star$ and $x_k$ with weight $\nu_{\star,k}=A_k-A_{k-1}$ \\(for $k=1,\hdots,N$)
        \[h(x_\star)\geq h(x_{k})+\inner{g_k}{x_\star-x_{k}}. \]
        \end{itemize}
		
		As previously linearly combining the previous inequalities leads to
		\begin{equation*}
		\begin{aligned}
		\sum_{k=1}^{N-1} \nu_{k,k+1} h(x_k)+\sum_{k=1}^{N} \nu_{\star,k} h(x_\star)
		\geq&\sum_{k=1}^{N-1}\nu_{k,k+1}[h(x_{k+1}) +\inner{g_{k+1}}{x_k-x_{k+1}}]\\
		&+\sum_{k=1}^N\nu_{\star,k}[h(x_k)+\inner{g_k}{x_\star-x_{k}}],
		\end{aligned}
		\end{equation*}
		which can be reformulated exactly as
		\begin{equation*}
		\begin{aligned}
		A_N(h(x_N)-h_\star) \leq& \tfrac{1}{4}\normsq{x_0-x_\star} - \tfrac14\normsq{x_\star-x_0+2\sum_{i=1}^N(A_i-A_{i-1}) g_i} \\
		\leq& \tfrac14\normsq{x_0-x_\star}.
		\end{aligned}
		\end{equation*}
		
		The tightness part of the proof is deferred to Appendix~\ref{sec:tight_orip}, where we show that the bound is satisfied with equality on one-dimensional linear minimization problems.
	\end{proof}

 A classical lower bound on the value of the sequence $\{A_k\}_k$ shows that the previous bound is a $O(N^{-2})$ when the $\lambda_k$ are lower bounded by some positive constant.
\begin{lemma}[Lemma 3.7 of~\cite{monteiro2013accelerated}]
Given a sequence $\{\lambda_k\}_k$ with $\lambda_k\geq 0$. Let $A_0 = 0$ and $A_{k+1} = A_k + \tfrac{\lambda_{k+1}+\sqrt{4\lambda_{k+1}A_{k}+\lambda_{k+1}^2}}{2}$ defined recursively, then 
\[A_k \geq \frac14\left(\sum_{i=1}^k\sqrt{\lambda_k}\right)^2 \quad \text{for } k\geq 1.\]
\end{lemma}

\begin{remark}
We emphasize that there is no constraint on the relation between primal and dual points outputted by the process hidden behind ``Obtain". In particular, primal-dual pairs of the form $(x_k,\partial h(x_k))$ or $(x_k,\tfrac{y_{k-1}-x_k}{\lambda_k})$ can be used.
\end{remark}


\section{Dealing with strongly convex objectives}\label{sec:str-pep}
In this section we present how the methodology detailed in Section~\ref{sec:pep} can be extended to support strongly convex functions. We illustrate it on the simple relatively inexact proximal method studied in Section~\ref{subsec:ppa} applied to strongly convex objectives.

For adjusting the performance estimation approach to strongly convex problems, we only need minor modifications. According to \cite[Corollary 2]{taylor2017smooth}, for $\mu >0$ and a set $S\in\Rd\times\Rd\times\R$
\begin{equation}\label{eq:interp-strcvx}
\begin{aligned}
\exists h\in\Fmu: \, &f=h(x),\quad g\in\partial h(x) \quad \forall (x,g,f)\in S \\ &\Leftrightarrow f'\geq f+\inner{g}{x'-x} + \tfrac{\mu}{2}\normsq{x-x'} \quad \forall (x,g,f),(x',g',f')\in S.
\end{aligned}
\end{equation}
In order to analyze inexact proximal minimization methods on strongly convex functions, we can simply follow Section~\ref{sec:pep} replacing the use of \eqref{eq:interp} by that of \eqref{eq:interp-strcvx}.

Let us illustrate that statement by instantiating the \emph{inexact proximal minimization algorithm} for strongly convex objectives. We recall the form of the updates
\begin{equation}\label{eq:example-str}
    \begin{aligned}
    w_{k+1} &= w_k - \lambda_{k+1}(v_{k+1}+e_{k+1})\\
    \normsq{e_{k+1}} &\leq \tfrac{\sigma^2}{\lambda_{k+1}^2}\normsq{w_{k+1}-w_k},
    \end{aligned}
\end{equation}
with $\{\lambda_k\}_k$ a sequence of nonnegative step sizes, $v_{k+1}\in \partial h (w_{k+1})$, $\{e_{k}\}_k$ a sequence of errors and $\sigma \in[0,1]$.

For $\mu > 0 $ we can study the following performance estimation problem for $\mu$-strongly convex objective functions $h$. In order to derive simpler worst-case guarantees, we use as slightly different initial condition compared with the previous section, which is $h(w_0)-h(w_\star) \leq R^2$ with $R\in \R^*$.

\begin{equation}\label{eq:pep-str}
\begin{aligned}
\max_{\substack{d,h\\w_\star,w_0,\hdots,w_{N}\in\Rd\\g_0,\hdots,g_{N}\in\Rd\\e_0,\hdots,e_{N}\in\Rd}}& h(w_{N})-h(w_\star)\\
\text{s.t. }&h\in\Fmu(\Rd), \quad w_\star\in\argmin_x h(x)\\
&h(w_0)-h(w_\star) \leq R^2\\
&w_1,\hdots,w_{N} \text{ satisfying } \eqref{eq:example-str}\\
&\normsq{e_k}\leq \tfrac{\sigma^2}{\lambda_k^2}\normsq{w_{k}-w_{k-1}} \quad k=1,\hdots,{N}.
\end{aligned}
\end{equation}

Following similar developments as those of Section \ref{sec:pep} and using interpolation conditions \eqref{eq:interp-strcvx} we get the semidefinite reformulation 
\begin{equation}\label{eq:sdp-ppa-str}
\begin{aligned}
\max_{G\succeq 0,\, H}\,\, &H({\bhh_N-\bhh_\star})&\\
\text{s.t. }&\;0\; \geq H(\bhh_i-\bhh_j)+\bvv_i^T G(\bww_j-\bww_i)\\
&\;\;\;\;\;+ \tfrac{\mu}{2}(\bww_j-\bww_i)^TG(\bww_j-\bww_i) \quad \forall i,j\in\{\star,0,\hdots,N\}\\
&R^2 \hspace{-0.04cm}\geq H(\bhh_0-\bhh_\star)\\
&\;0\; \geq \bee_i^TG\bee_i - \tfrac{\sigma^2}{\lambda_i^2}(\bww_i - \bww_{i-1})^TG(\bww_{i}-\bww_{i-1}) \quad \forall i\in\{1,\hdots,N\}.
\end{aligned}
\end{equation}
As before, we exhibit a dual feasible point, and the proof relies on weak duality.
\begin{theorem}\label{thm:ppa-str}
Let $\mu \geq 0$, $h\in\Fmu$, a sequence of step sizes $\{\lambda_k\}_k$ with $\lambda_k > 0$, and $\sigma\in[0,1]$. For any starting point $w_0\in\Rd$, $N\geq 1$, the iterates of \eqref{eq:example-str} satisfy
\[h(w_N)-h(w_\star)\leq \prod_{i=1}^N\left(\tfrac{1+\sigma}{1+\sigma+\lambda_k\mu}\right)^2(h(w_0)-h(w_\star)) , \]
with $w_\star \in \argmin_x h(x)$. Furthermore, this bound is tight: for all $\mu \geq 0$, $\{\lambda_k\}_k$ with $\lambda_k\geq0$, $\sigma\in[0,1]$, $d\in\mathbb{N}$, $w_0\in\mathbb{R}^d$, and $N\in\mathbb{N}$, there exists $h\in\Fmu(\Rd)$ such that this bound is achieved with equality.
\end{theorem}
\begin{proof}
Let us denote by $\rho_k = \tfrac{1+\sigma}{1+\sigma+\lambda_k \mu}\in[0,1]$ for $k=1,\hdots,N$.
We show the result by proving that 
\[h(w_{k}) - h(w_\star) \leq \rho_k^2(h(w_{k-1})-h(w_\star)) \quad k=1\hdots,N.\]
Indeed, chaining these inequalities for $k \in[1,N]$ leads to the desired conclusion.

We first detail the case $\sigma \in (0,1]$. Let $k\in [1,N]$, and consider the following inequalities with their corresponding weights :
\begin{itemize}
\item strong convexity between $w_{k-1}$ and $w_{k}$ with weight $\nu_{k-1,k}=\rho_k^{2}$
\[h(w_{k-1})\geq h(w_{k})+\inner{v_{k}}{w_{k-1}-w_{k}}+ \tfrac{\mu}{2}\normsq{w_{k-1}-w_k}, \]
\item strong convexity between $w_\star$ and $w_k$ with weight $\nu_{\star,k}=1-\rho_k^{2}$ 
\[h(w_\star)\geq h(w_{k})+\inner{v_k}{w_\star-w_{k}} + \tfrac{\mu}{2}\normsq{w_k-w_\star}, \]
\item approximation requirement on $w_k$ with weight $\nu_k=\tfrac{\lambda_k\rho_k}{2\sigma}$
\[\tfrac{\sigma^2}{\lambda_k^2}\normsq{w_{k-1}-w_k}\geq \normsq{e_k}.\]
\end{itemize}
By linearly combining previous inequalities, with their corresponding weights (which are nonnegative), we arrive to the following valid inequality:
\begin{equation}\label{eq:proof-str-ppa}
\begin{aligned}
    &\nu_{k-1,k}h(w_{k-1}) + \nu_{\star,k}h(w_\star) + \nu_k \tfrac{\sigma^2}{\lambda_k^2}\normsq{w_{k-1}-w_k} \\
    &\geq  \nu_{k-1,k}[h(w_k)+\inner{v_k}{w_{k-1}-w_k} + \tfrac{\mu}{2}\normsq{w_{k-1}-w_k}] \\
    &\,\,\,\,\,+ \nu_{\star,k}[h(w_k)+\inner{v_k}{w_\star-w_k}+\tfrac{\mu}{2}\normsq{w_\star-w_k}] + \nu_k\normsq{e_k}.
\end{aligned}    
\end{equation}

First we can regroup the function values together and observe that 
\begin{align*}
     &\nu_{k-1,k}h(w_k)+\nu_{\star,k}h(w_k)-\nu_{k-1,k}h(w_{k-1}) - \nu_{\star,k}h(w_\star)\\
    &= h(w_k)-h(w_\star) - \rho_k^2(h(w_{k-1})-h(w_\star)).
\end{align*}

Then, we regroup the vector variables together and use $w_{k} = w_{k-1}-\lambda_k(v_k+e_k)$ in 
\begin{align*}
    &\nu_{k-1,k}[\inner{v_k}{w_{k-1}-w_k} + \tfrac{\mu}{2}\normsq{w_{k-1}-w_k}]+ \nu_k[\normsq{e_k}-\tfrac{\sigma^2}{\lambda_k^2}\normsq{w_{k-1}-w_k}]  \\
    &\,\,\,\,\,+ \nu_{\star,k}[\inner{v_k}{w_\star-w_k}+\tfrac{\mu}{2}\normsq{w_\star-w_k}]\\
    &=\nu_{k-1,k}[\lambda_k\inner{v_k}{v_k+e_k} + \tfrac{\mu\lambda_k^2}{2}\normsq{v_k+e_k}]  + \nu_k[\normsq{e_k}-\sigma^2\normsq{v_k+e_k}]\\
    &\,\,\,\,\,+ \nu_{\star,k}[\inner{v_k}{w_\star-w_{k-1}+\lambda_k(v_k+e_k)}+\tfrac{\mu}{2}\normsq{w_\star-w_{k-1}+\lambda_k(v_k+e_k)}]\\
    &=\nu_{k-1,k}[\lambda_k\inner{v_k}{v_k+e_k} + \tfrac{\mu\lambda_k^2}{2}\normsq{v_k+e_k}]  + \nu_k[\normsq{e_k}-\sigma^2\normsq{v_k+e_k}]\\
    &\,\,\,\,\,+ \nu_{\star,k}[\tfrac{\mu}{2}\normsq{w_\star-w_{k-1}+\lambda_k(v_k+e_k)+\tfrac{1}{\mu}v_k} - \tfrac{1}{2\mu}\normsq{v_k}].
\end{align*}
We can then factorize the following expression 
\begin{align*}
    &\nu_{k-1,k}[\lambda_k\inner{v_k}{v_k+e_k} + \tfrac{\mu\lambda_k^2}{2}\normsq{v_k+e_k}]  + \nu_k[\normsq{e_k}-\sigma^2\normsq{v_k+e_k}]- \tfrac{\nu_{\star,k}}{2\mu}\normsq{v_k}\\
    &=[\nu_{k-1,k}\tfrac{\mu\lambda_k^2}{2}+\nu_k(1-\sigma^2)]\normsq{e_k}
    + [\nu_{k-1,k}\lambda_k+\nu_{k-1,k}\mu\lambda_k^2-2\nu_k\sigma^2]\inner{e_k}{v_k}\\
    &\,\,\,\,\,+[\nu_{k-1,k}\lambda_k + \nu_{k-1,k}\tfrac{\mu\lambda_k^2}{2}-\nu_k\sigma^2-\tfrac{\nu_{\star,k}}{2\mu}]\normsq{v_k}\\
    &=\tfrac{\lambda_k\rho_k}{2\sigma}[\lambda_k\mu\sigma\rho_k + (1-\sigma^2)]\normsq{e_k}
    + \lambda_k\rho_k[\rho_k(1+\lambda_k\mu) - \sigma]\inner{e_k}{v_k}+[\tfrac{(\rho_k+\lambda_k\mu\rho_k)^2 -(1+\lambda_k\mu\sigma \rho_k)}{2\mu}]\normsq{v_k}\\
    & = \tfrac{\lambda_k(1+\sigma)^2(1-\sigma^2+\lambda_k\mu)}{2\sigma(1+\sigma +\lambda_k\mu)^2}\normsq{e_k}
    + \tfrac{\lambda_k(1+\sigma)(1-\sigma^2+\lambda_k\mu)}{(1+\sigma + \lambda_k\mu)^2}\inner{e_k}{v_k}+\tfrac{\lambda_k\sigma(1-\sigma^2+\lambda_k\mu)}{2(1+\sigma+\lambda_k\mu)^2}\normsq{v_k}\\
    & = \tfrac{\lambda_k(1+\sigma)^2(1-\sigma^2+\lambda_k\mu)}{2\sigma(1+\sigma +\lambda_k\mu)^2}\normsq{e_k+\tfrac{\sigma}{1+\sigma}v_k},
\end{align*}
where we replaced $\rho_k$ by its expression in the second to last line.
Finally \eqref{eq:proof-str-ppa} can be written as 
\begin{align*}
    0 \geq&\, h(w_k)-h(w_\star) - \rho_k^2(h(w_{k-1})-h(w_\star)) \\
    &+ \tfrac{(1-\rho_k^2)\mu}{2}\normsq{w_\star-w_{k-1}+\lambda_k(v_k+e_k)+\tfrac{1}{\mu}v_k}\\
    &+ \tfrac{\lambda_k(1+\sigma)^2(1-\sigma^2+\lambda_k\mu)}{2\sigma(1+\sigma +\lambda_k\mu)^2}\normsq{e_k+\tfrac{\sigma}{1+\sigma}v_k}.
\end{align*}
Since $\rho_k,\sigma \in [0,1] $ the leading factors in front of the squared Euclidean norms are nonnegative and this leads to 
\[  h(w_k)-h(w_\star) \leq \rho_k^2(h(w_{k-1})-h(w_\star)) \]
which concludes the first part of the proof for $\sigma\in(0,1]$, according to our initial remark.

For the exact case (i.e., $\sigma =0$), the proof carries on likewise, by only combining the first two inequalities, encoding strong convexity, leading to
\begin{align*}
    0 \geq&\, h(w_k)-h(w_\star) - \rho_k^2(h(w_{k-1})-h(w_\star)) + \tfrac{(1-\rho_k^2)\mu}{2}\normsq{w_\star-w_{k-1}+\left(\lambda_k+\tfrac{1}{\mu}\right)v_k},
\end{align*}
and the desired conclusion follows.
The tightness part is deferred to Appendix~\ref{app:tight-str} where we show that the bound is satisfied with equality on a simple quadratic function.
\end{proof}

\section{Conclusion}\label{sec:ccl} In this work, we showed that the performance estimation framework, initiated by Drori and Teboulle~\cite{drori2014performance}, allows studying first-order methods involving natural notions of inexact proximal operations. On the way, we reviewed natural accuracy requirements used in the literature for characterizing inexact proximal operations. We also used the approach for optimizing the parameters of an inexact proximal point algorithm. Finally, we presented a simple extension to the strongly convex setting. 

As future works, we believe the approach can be extended to inexact Bregman proximal steps (see e.g.,~\cite{eckstein1998approximate}), and to inexact resolvent for monotone operators (see e.g.,~\cite{solodov1999hybrid}), for example by following steps taken~\cite{dragomir2021optimal,ryu2018operator}. Further using those tools for designing optimized methods involving inexact proximal operations for monotone inclusions, and variational inequalities are also possibilities. Let us also mention that it is currently unclear to us whether similar techniques can be used for studying higher-order proximal methods, as recently introduced by Nesterov~\cite{nesterov2020inexactAcc,nesterov2020inexact}.

Finally, an alternate, and more geometric, approach for studying inexact proximal operations could be to extend \emph{scaled relative graphs}~\cite{ryu2019scaled} to deal with inaccuracies.

\paragraph{Codes} Codes, that include notebooks for helping the reader reproducing the proofs and implementation of the performance estimation problems, are available at
\begin{center}
\url{https://github.com/mathbarre/InexactProximalOperators/tree/version-2}
\end{center}
Notions of inexactness were also included in the performance estimation toolbox~\cite{taylor2017performance}.
\paragraph{Acknowledgements} The authors would like to thank Ernest Ryu for insightful feedbacks on a preliminary version of this manuscript. MB acknowledges support from an AMX fellowship. The authors acknowledge support from the European Research Council (grant SEQUOIA 724063). This work was funded in part by the French government under management of Agence Nationale
de la Recherche as part of the ``Investissements d'avenir" program, reference ANR-19-P3IA-0001
(PRAIRIE 3IA Institute).
\bibliographystyle{spmpsci}      
\bibliography{bib}
\appendix
\normalsize
\section{More examples of fixed-step inexact proximal methods}\label{app:methods}
This extends the list of examples of Section~\ref{sss:examples_algo}. 
\begin{itemize}
    \item The \emph{hybrid approximate extragradient algorithm} (see \cite{solodov1999hybrid} or \cite[Section 4]{monteiro2010complexity}) can be described as
    \[ x_{k+1}=x_k-\eta_{k+1}g_{k+1},\]
    such that $\exists u_{k+1}, \PDg_{\eta_{k+1} h}(u_{k+1},g_{k+1};\,x_{k}) \leq \tfrac{\sigma^2}{2}\normsq{u_{k+1}-x_{k}}$ (see Lemma~\ref{lem:pdgag-eps} for a link between $\varepsilon$-subgradient formulation and primal-dual gap).
    One iteration of this form can be artificially cast into three iterations of \eqref{eq:generic-prox} as 
    \[\left \{ \begin{array}{rcl}
         w_{3k+1} &=& w_{3k}-e_{3k}\\
         w_{3k+2} &=& w_{3k+1}-e_{3k+1}\\
         w_{3k+3} &=& w_{3k+2} +e_{3k}+e_{3k+1} - \eta_{k+1}v_{3k+2} \\
    \end{array}\right. \]
    with  $v_{3k+2} \in \partial h(w_{3k+2})$ . This corresponds to setting $\lambda_{3k+1} = \lambda_{3k+2} =\lambda_{3k+3} = 0$,  $\alpha_{3k+3,3k+2}=\eta_{k+1}$, $\beta_{3k+1,3k} = \beta_{3k+2,3k+1}=1$, $\beta_{3k+3,3k+1}=\beta_{3k+3,3k+2} = -1$ and the other parameters to zero. Notice that $w_{3k+3} = w_{3k} - \eta_{k+1}v_{3k+2}$. By requiring $\PDg_{\eta_{k+1} h}(w_{3k+1},v_{3k+2};\,w_{3k}) \leq \tfrac{\sigma^2}{2}\normsq{w_{3k+1}-w_{3k}}$ we can identify the primal-dual pair $(u_{k+1},g_{k+1})$ with $(w_{3k+1},v_{3k+2})$ and iterates $x_{k+1}$ with $w_{3k+3}$.
    In addition, we set 
    \begin{align*}
        \text{EQ}_{3k+1} & =0,\\
        \text{EQ}_{3k+2} & =0,\\
        \text{EQ}_{3k+3} &= \PDg_{\eta_{k+1} h}(w_{3k+1},v_{3k+2};\,w_{3k}) - \tfrac{\sigma^2}{2}\normsq{w_{3k+1}-w_{3k}}.
    \end{align*}
Using $v_{3k+2}\in \partial h (w_{3k+2})$, we have $h^*(v_{3k+2}) = \inner{v_{3k+2}}{w_{3k+2}} - h(w_{3k+2})$ and thus 
    \begin{align*}
        \text{EQ}_{3k+3} =& \tfrac12\normsq{w_{3k+1}-w_{3k+3}}+\eta_{k+1}(h(w_{3k+1})-h(w_{3k+2}) \\
        &- \inner{v_{3k+2}}{w_{3k+1}-w_{3k+2}})-\tfrac{\sigma^2}{2}\normsq{w_{3k+1}-w_{3k}},
    \end{align*} which complies with \eqref{eq:generic_inexact} and is Gram-representable.
\item The \emph{inexact accelerated proximal point algorithm} \emph{IAPPA1} in its form from \cite[Section 5]{salzo2012inexact} can be written as 
    \[\left \{ \begin{array}{rcl}
         t_{k+1}&=& \tfrac{1+\sqrt{1+4t_k^2\tfrac{\eta_{k+1}}{\eta_{k+2}}}}{2}  \\
         x_{k+1} &=& y_{k} - \eta_{k+1}(g_{k+1}+r_{k+1}) \\
         y_{k+1} &=& x_{k+1} + \tfrac{t_k-1}{t_{k+1}}(x_{k+1}-x_k)
    \end{array}\right. \]
    with $t_0 = 1$, $\{\eta_k\}_k$ a sequence of step sizes, $y_0=x_0\in \Rd$ along with an inexactness criterion of the form $\PDg_{\eta_{k+1} h}(x_{k+1},g_{k+1};y_k) \leq \varepsilon_{k+1}$ given a nonnegative sequence $\{\varepsilon_k\}_k$. Similarly to Güler's method we get the recursive formulation
    \[x_{k+2} = \left(1+\tfrac{t_k-1}{t_{k+1}}\right)x_{k+1} - \tfrac{t_k-1}{t_{k+1}}x_k - \eta_{k+2}(g_{k+2}+r_{k+2}). \]
    We consider particular iterations from \eqref{eq:generic-prox} of the form 
    \[\left \{ \begin{array}{rcl}
         w_{2k+1} &=& w_{2k} - e_{2k} \\
         w_{2k+2} &=& w_{2k+1} -\displaystyle\sum_{i=1}^{2k+1}\alpha_{2k+2,i}v_{i} -\sum_{i=0}^{2k+1}\beta_{2k+2,i}e_i,
    \end{array}\right. \]
    with initial iterate $w_0=x_0$. We aim at finding parameters $\alpha_{i,j}$, $\beta_{i,j}$ such that we can identify $\{w_{2k}\}_k$ with $\{x_k\}_k$ (i.e., any sequence $\{x_k\}_k$ can be obtained as a sequence $\{w_{2k}\}_k$).
    We set $\alpha_{2k+2,2k+1}=\beta_{2k+2,2k+1} = \eta_{k+1}$, $\alpha_{2k+2,i} = \tfrac{t_{k-1}-1}{t_k}\alpha_{2k,i}$ for $i=1,\hdots,2k-1$ and $\beta_{2k+2,i} = \tfrac{t_{k-1}-1}{t_k}\beta_{2k,i}$ for $i\in\{0,\hdots,2k-1\}\backslash\{2(k-1)\}$ as well as $\beta_{2k+2,2k} = -1$ and $\beta_{2k+2,2(k-1)} = \tfrac{t_{k-1}-1}{t_k}(1+\beta_{2k,2(k-1)})$.
    
    This gives 
    \begin{align*}
        w_{2(k+1)} =& w_{2k+1} + e_{2k} - \tfrac{t_{k-1}-1}{t_k}(e_{2(k-1)})  -\tfrac{t_{k-1}-1}{t_k}\displaystyle\sum_{i=1}^{2k-1}\alpha_{2k,i}v_{i} \\
        &-\tfrac{t_{k-1}-1}{t_k}\sum_{i=0}^{2k-1}\beta_{2k,i}e_i 
        - \eta_{k+1}(v_{2k+1}+e_{2k+1})\\
        =& (1+\tfrac{t_{k-1}-1}{t_k})w_{2k} -\tfrac{t_{k-1}-1}{t_k}w_{2(k-1)}  - \eta_{k+1}(v_{2k+1}+e_{2k+1}),
    \end{align*} 
    which shows that $\{w_{2k}\}_k$ follows the same recursive equation as $\{x_{k}\}_k$. In addition, we have $w_0 = x_0$ and $w_2 = x_0 - \eta_{1}(v_1+e_1)$ similar to $x_1 = x_0 -\eta_1(g_1+r_1)$.
    Requiring $\PDg_{\eta_{k+1} h}(w_{2k+2},v_{2k+1};\,w_{2k+2}+\eta_{k+1}(v_{2k+1}+e_{2k+1})) \leq \varepsilon_{k+1}$ (with the convention $w_{-1}=w_0$) allows to identify the primal-dual pair $(x_{k+1},g_{k+1})$ with $(w_{2k+2},v_{2k+1})$.
    
    Finally, we can set $\text{EQ}_{2k+2} = \PDg_{\eta_{k+1} h}(w_{2k+2},v_{2k+1};\,w_{2k+2}+\eta_{k+1}(v_{2k+1}+e_{2k+1})) - \varepsilon_{k+1}$ which is Gram-representable (similar to \emph{hybrid approximate extragradient algorithm}).
    
    Note that we can proceed similarly for \emph{IAPPA2} from \cite[Section 5]{salzo2012inexact} with sequence $\{a_k\}_k$ constant equal to $1$, by removing the sequence $\{r_k\}_k$ ``type 2" errors).
    
    \item The \emph{acceleration hybrid proximal extragradient algorithm} (A-HPE) from \cite[Section 3]{monteiro2013accelerated} can be written as 
    \[\left \{
    \begin{array}{rcl}
    a_{k+1} &=& \tfrac{\eta_{k+1}+\sqrt{\eta_{k+1}^2+4\eta_{k+1}A_k}}{2}\\
    A_{k+1} &=& A_k + a_{k+1}\\
    \tilde{x}_k &=& y_k + \tfrac{a_{k+1}}{A_{k+1}}(x_k-y_k)\\
    y_{k+1} &=& \tilde{x}_k - \eta_{k+1}(g_{k+1}+r_{k+1})\\
    x_{k+1} &=& x_k - a_{k+1}g_{k+1},
    \end{array}\right. \]
    with $A_0=0$, $\{\eta_k\}_k$ a sequence of step sizes, $y_0=x_0\in \Rd$ along with an inexactmess criterion of the form $\PDg_{\eta_{k+1} h}(y_{k+1},g_{k+1};\tilde{x}_k) \leq \tfrac{\sigma}{2}\normsq{y_{k+1}-\tilde{x}_k}$ given a parameter $\sigma \in [0,1]$. As in the previous examples, we search for a recursive equation followed by the sequence $\{y_k\}_k$. By performing multiple substitutions, we obtain
    \begin{align*}
        y_{k+2}=& \tilde{x}_{k+1} - \eta_{k+2}(g_{k+2}+r_{k+2})\\
        =& \tfrac{A_{k+1}}{A_{k+2}}y_{k+1} + \tfrac{a_{k+2}}{A_{k+2}}x_{k+1}- \eta_{k+2}(g_{k+2}+r_{k+2})\\
        =& \tfrac{A_{k+1}}{A_{k+2}}y_{k+1} + \tfrac{a_{k+2}}{A_{k+2}}\left(x_{k}- a_{k+1}g_{k+1}\right)- \eta_{k+2}(g_{k+2}+r_{k+2})\\
        =& \tfrac{A_{k+1}}{A_{k+2}}y_{k+1} + \tfrac{a_{k+2}}{A_{k+2}}\left(\tfrac{A_{k+1}}{a_{k+1}}\tilde{x}_k - \tfrac{A_k}{a_{k+1}}y_{k} - a_{k+1}g_{k+1}\right)- \eta_{k+2}(g_{k+2}+r_{k+2})\\
        =& \tfrac{A_{k+1}}{A_{k+2}}y_{k+1} + \tfrac{a_{k+2}}{A_{k+2}}\left(\tfrac{A_{k+1}}{a_{k+1}}\left(y_{k+1}+\eta_{k+1}(g_{k+1}+r_{k+1}) \right)- \tfrac{A_k}{a_{k+1}}y_{k} - a_{k+1}g_{k+1}\right)\\
        &- \eta_{k+2}(g_{k+2}+r_{k+2})\\
        =& \left(\tfrac{A_{k+1}}{A_{k+2}}+\tfrac{a_{k+2}A_{k+1}}{A_{k+2}a_{k+1}}\right)y_{k+1}-\tfrac{a_{k+2}A_k}{A_{k+2}a_{k+1}}y_k + \tfrac{a_{k+2}}{A_{k+2}}\left(\tfrac{A_{k+1}}{a_{k+1}}\eta_{k+1} - a_{k+1}\right)g_{k+1}\\
        &+\tfrac{a_{k+2}A_{k+1}}{A_{k+2}a_{k+1}}\eta_{k+1}r_{k+1}- \eta_{k+2}(g_{k+2}+r_{k+2})\\
        =& \left(1+\tfrac{a_{k+2}A_{k}}{A_{k+2}a_{k+1}}\right)y_{k+1}-\tfrac{a_{k+2}A_k}{A_{k+2}a_{k+1}}y_k + \tfrac{a_{k+2}}{A_{k+2}}\left(\tfrac{A_{k+1}}{a_{k+1}}\eta_{k+1} - a_{k+1}\right)g_{k+1}\\
        &+\tfrac{a_{k+2}A_{k+1}}{A_{k+2}a_{k+1}}\eta_{k+1}r_{k+1}- \eta_{k+2}(g_{k+2}+r_{k+2}).
    \end{align*}
    Similar to \emph{IAPPA1}, we consider particular iterations from \eqref{eq:generic-prox} of the form 
    \[\left \{ \begin{array}{rcl}
         w_{2k+1} &=& w_{2k} - e_{2k} \\
         w_{2k+2} &=& w_{2k+1} -\displaystyle\sum_{i=1}^{2k+1}\alpha_{2k+2,i}v_{i} -\sum_{i=0}^{2k+1}\beta_{2k+2,i}e_i,
    \end{array}\right. \]
    with initial iterate $w_0=x_0$. We aim at finding parameters $\alpha_{i,j}$, $\beta_{i,j}$ such that we can identify $\{w_{2k}\}_k$ with $\{y_k\}_k$ (i.e., any sequence $\{y_k\}_k$ can be obtained as a sequence $\{w_{2k}\}_k$). We set $\alpha_{2(k+1),2k+1}=\beta_{2(k+1),2k+1} = \eta_{k+1}$, $\alpha_{2(k+1),i} = \tfrac{a_{k+1}A_{k-1}}{A_{k+1}a_{k}}\alpha_{2k,i}$ for $i\in\{1,\hdots,2(k-1)\}$ and $\beta_{2(k+1),i} = \tfrac{a_{k+1}A_{k-1}}{A_{k+1}a_{k}}\beta_{2k,i}$ for $i\in\{0,\hdots,2k-3\}$ as well as $\beta_{2(k+1),2k} = -1$, $\beta_{2(k+1),2k-1} = \tfrac{a_{k+1}}{A_{k+1}a_{k}}(A_{k-1}\beta_{2k,2k-1} -A_k\eta_{k} )$, $\beta_{2(k+1),2(k-1)} = \tfrac{a_{k+1}A_{k-1}}{A_{k+1}a_{k}}(1+\beta_{2k,2(k-1)})$ and $\alpha_{2(k+1),2k-1} = \tfrac{a_{k+1}}{A_{k+1}}\left(\tfrac{A_{k-1}}{a_k}\alpha_{2k,2k-1}-\tfrac{A_{k}}{a_{k}}\eta_{k} + a_{k}\right) $.
    
    This gives 
    \begin{align*}
        w_{2(k+1)} = & w_{2k+1} +e_{2k} +\tfrac{a_{k+1}A_k}{A_{k+1}a_{k}}\eta_{k}e_{2k-1} - \tfrac{a_{k+1}A_{k-1}}{A_{k+1}a_k}e_{2(k-1)}+ \tfrac{a_{k+1}}{A_{k+1}}\left(\tfrac{A_{k}}{a_{k}}\eta_{k} - a_{k}\right)v_{2k-1}\\ &-\tfrac{a_{k+1}A_{k-1}}{A_{k+1}a_{k}}\sum_{i=1}^{2k-1}\alpha_{2k,i}v_i - \tfrac{a_{k+1}A_{k-1}}{A_{k+1}a_{k}}\sum_{i=0}^{2k-1}\beta_{2k,i}e_i - \eta_{n+1}(v_{2k+1}+e_{2k+1})\\
        = & w_{2k} +\tfrac{a_{k+1}A_k}{A_{k+1}a_{k}}\eta_{k}e_{2k-1} - \tfrac{a_{k+1}A_{k-1}}{A_{k+1}a_k}e_{2(k-1)}+ \tfrac{a_{k+1}}{A_{k+1}}\left(\tfrac{A_{k}}{a_{k}}\eta_{k} - a_{k}\right)v_{2k-1}\\ &+\tfrac{a_{k+1}A_{k-1}}{A_{k+1}a_{k}}(w_{2k}-w_{2(k-1)}+e_{2(k-1)}) - \eta_{n+1}(v_{2k+1}+e_{2k+1})\\
        = & \left(1+\tfrac{a_{k+1}A_{k-1}}{A_{k+1}a_{k}}\right)w_{2k}-\tfrac{a_{k+1}A_{k-1}}{A_{k+1}a_{k}}w_{2(k-1)}+\tfrac{a_{k+1}}{A_{k+1}}\left(\tfrac{A_{k}}{a_{k}}\eta_{k} - a_{k}\right)v_{2k-1}\\
        &+\tfrac{a_{k+1}A_k}{A_{k+1}a_{k}}\eta_{k}e_{2k-1} - \eta_{n+1}(v_{2k+1}+e_{2k+1}),
    \end{align*}
    which shows that $\{w_{2k}\}_k$ follows the same recursive equation as $\{y_{k}\}_k$. In addition, we have $w_0 = x_0 = y_0$ and $w_2 = y_0 - \eta_{1}(v_1+e_1)$ similar to $x_1 = x_0 -\eta_1(g_1+r_1)$.
    Requiring $\PDg_{\eta_{k+1} h}(w_{2(k+1)},v_{2k+1};\,w_{2k+2}+\eta_{k+1}(v_{2k+1}+e_{2k+1})) \leq \tfrac{\sigma^2}{2}\normsq{w_{2(k+1)}-w_{2k}}$ allows to identify the primal-dual pair $(y_{k+1},g_{k+1})$ with $(w_{2(k+1)},v_{2k+1})$.
    
    Finally, we set $\text{EQ}_{2k+2} = \PDg_{\eta_{k+1} h}(w_{2k+2},v_{2k+1};\,w_{2k+2}+\eta_{k+1}(v_{2k+1}+e_{2k+1})) - \tfrac{\sigma^2}{2}\normsq{w_{2(k+1)}-w_{2k}}$ which is Gram-representable (similar to \emph{hybrid approximate extragradient algorithm}).
\end{itemize}

\section{Interpolation with {$\mathbf{\varepsilon}$}-subdifferentials}\label{app:epssubdiff}
In this section, we provide the necessary interpolation result for working with $\varepsilon$-subdifferentials inside performance estimation problems.
\begin{theorem} Let $I$ be a finite set of indices and $S=\{(w_i,v_i,h_i,\varepsilon_i)\}_{i\in I}$  with $w_i,v_i\in\Rd$, $h_i,\varepsilon_i\in\mathbb{R}$ for all $i\in I$. There exists $h\in\Fccp(\Rd)$ satisfying 
\begin{equation}\label{eq:interp_setting}
h_i = h(w_i),\text{ and } v_i\in\partial_{\varepsilon_i}h(w_i) \text{ for all } i\in I
\end{equation}
if and only if
\begin{equation}\label{eq:interp_eps}
\begin{aligned}
h_i\geq h_j +\inner{v_j}{w_i-w_j}-\varepsilon_j
\end{aligned}
\end{equation}
holds for all $i,j\in I$.
\end{theorem}
\begin{proof} $(\Rightarrow)$ Assuming $h\in\Fccp$ and~\eqref{eq:interp_setting}, the inequalities~\eqref{eq:interp_eps} hold by definition.

$(\Leftarrow)$ Assuming~\eqref{eq:interp_eps} hold, one can perform the following construction:
\begin{equation*}
\begin{aligned}
\tilde{h}(x)=\max_i\{h_i+\inner{v_i}{x-w_i}-\varepsilon_i\},
\end{aligned}
\end{equation*}
and one can easily check that $h=\tilde{h}\in\Fccp$ satisfies~\eqref{eq:interp_setting}.
\end{proof}

\section{Equivalence with Güler's method}\label{app:guler-eq}
In this section, we show that optimized algorithm \eqref{eq:orip} and Güler's second method \cite[Section 6]{guler1992new} are equivalent (i.e., produce the same iterates), in the case of exact proximal computations (i.e., $\sigma=0$).

We consider a constant sequence of step sizes $\{\lambda_k\}_k$ with $\lambda_k = \lambda >0$.
In Güler's second method, the sequence $\{\beta_k\}_k$ is defined as 
$\beta_1 = 1$ and \[ \beta_{k+1} = \tfrac{1+\sqrt{4\beta_k^2+1}}{2}.\]
The sequence $\{A_k\}_k$ generated by \eqref{eq:orip} satisfies $A_0=0$ and
\[A_{k+1} = A_k+\tfrac{\lambda+\sqrt{4\lambda A_k + \lambda^2}}{2},\quad k\geq 0. \]
We can link together these two sequences through the following equality
\begin{equation}\label{eq:equi-beta-A}
    \beta_{k} = \tfrac{A_{k}-A_{k-1}}{\lambda}, \quad k\geq 1.
\end{equation}
Let us prove it recursively. First, observe that $\beta_1 = 1$ and $\tfrac{A_1-A_0}{\lambda} = 1$. Then assuming that the property is true for some $k \geq 1$,
we have 
\begin{align*}
    \beta_{k+1} &= \tfrac{1+\sqrt{4\beta_k^2 +1}}{2}\\
    &= \tfrac{1+\sqrt{4\tfrac{(A_{k+1}-A_k)^2}{\lambda^2} +1}}{2}.
\end{align*}
One might notice that \[(A_{k+1}-A_k)^2 = \tfrac{2\lambda^2+4\lambda A_k+2\lambda\sqrt{4\lambda A_k + \lambda^2}}{4}= \lambda A_{k+1},\]
which gives 
\begin{align*}
    \beta_{k+1} &= \tfrac{1+\sqrt{4\tfrac{A_{k+1}}{\lambda} +1}}{2}\\
    &= \tfrac{\lambda + \sqrt{4\lambda A_{k+1} + \lambda^2}}{2\lambda}\\
    &= \tfrac{A_{k+2}-A_{k+1}}{\lambda},
\end{align*}
and we finally arrive to~\eqref{eq:equi-beta-A}.

In the exact case ($\sigma =0$) iterations of \eqref{eq:orip} can be written as
\[\left\{\begin{array}{ccl}
    y_k &=& x_k + \tfrac{\lambda}{A_{k+1}-A_k}(z_k-x_k)  \\
    x_{k+1} &=& \prox_{\lambda h}(y_k)\\
    z_{k+1} &=& z_k +\tfrac{2(A_{k+1}-A_k)}{\lambda}(x_{k+1}-y_{k}).
\end{array}\right.\]

Therefore, we can express
\begin{align*}
    y_{k+1} & = x_{k+1} + \tfrac{\lambda}{A_{k+2}-A_{k+1}}(z_k +\tfrac{2(A_{k+1}-A_k)}{\lambda}(x_{k+1}-y_{k})-x_{k+1})\\
    &=x_{k+1} + \tfrac{\lambda}{A_{k+2}-A_{k+1}}(x_k - \tfrac{A_{k+1}-A_k}{\lambda}(x_k-y_k) +\tfrac{2(A_{k+1}-A_k)}{\lambda}(x_{k+1}-y_{k})-x_{k+1})\\
    &=x_{k+1} + \tfrac{\lambda}{A_{k+2}-A_{k+1}}\left(\left(\tfrac{A_{k+1}-A_k}{\lambda}-1\right)(x_{k+1}-x_k) +\tfrac{A_{k+1}-A_k}{\lambda}(x_{k+1}-y_{k})\right),
\end{align*}
and combining the last equality with \eqref{eq:equi-beta-A} leads to 
\[y_{k+1} = x_{k+1} + \tfrac{\beta_{k+1}-1}{\beta_{k+2}}(x_{k+1}-x_k) + \tfrac{\beta_{k+1}}{\beta_{k+2}}(x_{k+1}-y_k) \]
which is exactly the update in Güler's second method \cite[Section 6]{guler1992new} modulo a translation in the indices of the $\{y_k\}_k$ sequence (indeed in Güler's method $y_1=x_0$ whereas in \eqref{eq:orip} $y_0=x_0$).

\section{Missing details in Theorem~\ref{thm:orip2}}\label{sec:missingparts_orip}
The missing elements in the proof of Theorem~\ref{thm:orip2} are presented bellow.
\begin{proof}
Let us rewrite the method in terms of a single sequence, by substitution of $y_k$ and $z_k$:
\begin{equation}\label{eq:eq_orip}
\begin{aligned}
e_k &\coloneqq \tfrac{1}{\lambda_k}\left(y_{k-1} -\lambda_kg_k - x_k\right)\\
x_{k}&=\tfrac{\lambda_k}{A_{k}-A_{k-1}}\left(x_0-\tfrac{2}{1+\sigma}\sum_{i=1}^{k-1}(A_{i}-A_{i-1}) g_i\right)+\left(1-\tfrac{\lambda_k}{A_{k}-A_{k-1}}\right)x_{k-1} -\lambda_{k} (g_{k}+e_{k}),\\
\end{aligned}
\end{equation}
and let us state the following identity on the $A_k$ coefficients
\begin{equation}\label{eq:Ak}
    \lambda_{k+1}A_{k+1}=(A_{k+1}-A_k)^2 \text{ (for $k\geq 0$)}.
\end{equation}

We prove the desired convergence result by induction. First, for $N=1$
\begin{equation*}
    \begin{aligned}
    0 \geq& \nu_{\star,1}[h(u_1)-h(x_\star) + \inner{g_1}{x_\star-u_1}] + \nu_{1,1}[h(u_1)-h(x_1)+\inner{g_1}{x_1-u_1}] \\
    & + \nu_1[\tfrac{\lambda_1}{2}\normsq{e_1}-\tfrac{\lambda_1\sigma^2}{2}\normsq{e_1+g_1} + h(x_1) - h(u_1) -\inner{g_1}{x_1-u_1}]
    \end{aligned}
\end{equation*}
with $\nu_{\star,1}= \tfrac{A_1-A_0}{1+\sigma}=\tfrac{A_1}{1+\sigma}$ as $A_0=0$, $\nu_{1,1} = \tfrac{(1-\sigma)A_1}{\sigma(1+\sigma)}$ and $\nu_1 = \tfrac{A_1}{\sigma(1+\sigma)}$. This gives 
\begin{align*}
0 \geq& \tfrac{A_1}{1+\sigma}(h(x_1)-h_\star) + \tfrac{A_1}{1+\sigma}\inner{g_1}{x_\star-x_1} + \tfrac{A_1}{\sigma(1+\sigma)}[\tfrac12\normsq{e_1}-\tfrac{\sigma^2}{2}\normsq{e_1+g_1}]\\
=& \tfrac{A_1}{1+\sigma}(h(x_1)-h_\star) + \tfrac{A_1}{1+\sigma}\inner{g_1}{x_\star-x_0 + \lambda_1(g_1+e_1)} + \tfrac{A_1}{\sigma(1+\sigma)}[\tfrac{\lambda_1}{2}\normsq{e_1}-\tfrac{\lambda_1\sigma^2}{2}\normsq{e_1+g_1}]\\
=& \tfrac{A_1}{1+\sigma}(h(x_1)-h_\star) + \tfrac{1}{2}\inner{2\tfrac{A_1}{1+\sigma}g_1}{x_\star-x_0}  + \inner{\tfrac{A_1}{1+\sigma}g_1}{\lambda_1(g_1+e_1)} + \tfrac{A_1}{\sigma(1+\sigma)}[\tfrac{\lambda_1}{2}\normsq{e_1}-\tfrac{\lambda_1\sigma^2}{2}\normsq{e_1+g_1}]\\
=& \tfrac{A_1}{1+\sigma}(h(x_1)-h_\star) + \tfrac14\normsq{x_\star-x_0+2\tfrac{A_1}{1+\sigma}g_1} -\tfrac14\normsq{x_\star-x_0} -\normsq{\tfrac{A_1}{1+\sigma}g_1}  \\ &+\inner{\tfrac{A_1}{1+\sigma}g_1}{\lambda_1(g_1+e_1)} + \tfrac{A_1}{\sigma(1+\sigma)}[\tfrac{\lambda_1}{2}\normsq{e_1}-\tfrac{\lambda_1\sigma^2}{2}\normsq{e_1+g_1}]\\
=& \tfrac{A_1}{1+\sigma}(h(x_1)-h_\star) + \tfrac14\normsq{x_\star-x_0+2\tfrac{A_1}{1+\sigma}g_1} -\tfrac14\normsq{x_\star-x_0} + \tfrac{A_1\lambda_1(1-\sigma)}{2\sigma}\normsq{e_1} \\
&+ \tfrac{A_1\lambda_1(1-\sigma)}{1+\sigma}\inner{g_1}{e_1} + \tfrac{A_1}{1+\sigma}\left(-\tfrac{A_1}{1+\sigma} + \lambda_1 - \tfrac{\lambda_1\sigma}{2} \right)\normsq{g_1} \\
 =& \tfrac{A_1}{1+\sigma}(h(x_1)-h_\star) + \tfrac14\normsq{x_\star-x_0+2\tfrac{A_1}{1+\sigma}g_1} -\tfrac14\normsq{x_\star-x_0} \\
 &+ \tfrac{A_1\lambda_1(1-\sigma)}{2\sigma}\normsq{e_1+\tfrac{\sigma}{1+\sigma}g_1} 
 + \tfrac{A_1}{1+\sigma}\left(-\tfrac{A_1}{1+\sigma} + \lambda_1 - \tfrac{\lambda_1\sigma}{2} -\tfrac{\lambda_1(1-\sigma)\sigma}{2(1+\sigma)}\right)\normsq{g_1} \\
 =& \tfrac{A_1}{1+\sigma}(h(x_1)-h_\star) + \tfrac14\normsq{x_\star-x_0+2\tfrac{A_1}{1+\sigma}g_1} -\tfrac14\normsq{x_\star-x_0} \\
 &+ \tfrac{A_1\lambda_1(1-\sigma)}{2\sigma}\normsq{e_1+\tfrac{\sigma}{1+\sigma}g_1} 
 + \tfrac{A_1}{1+\sigma}\left(\tfrac{\lambda_1-A_1}{1+\sigma}\right)\normsq{g_1} \\
 =& \tfrac{A_1}{1+\sigma}(h(x_1)-h_\star) + \tfrac14\normsq{x_\star-x_0+2\tfrac{A_1}{1+\sigma}g_1} -\tfrac14\normsq{x_\star-x_0} + \tfrac{A_1\lambda_1(1-\sigma)}{2\sigma}\normsq{e_1+\tfrac{\sigma}{1+\sigma}g_1}, 
\end{align*}
where we used in the last line that $A_1 = \lambda_1$.

Now, assuming the weighted sum can be reformulated as the desired inequality for $N=k$, that is:
\begin{align*}
0\geq& \tfrac{A_k}{1+\sigma}(h(x_k)-h_\star)  -\tfrac14\normsq{x_\star-x_0} + \tfrac14\normsq{x_\star-x_0+\tfrac{2}{1+\sigma}\sum_{i=1}^{k}(A_{i}-A_{i-1})g_i} \\
 &+ \tfrac{(1-\sigma)}{2\sigma}\sum_{i=1}^k A_i\lambda_i\normsq{e_i+\tfrac{\sigma}{1+\sigma}g_i}, 
\end{align*}
let us prove it also holds true for $N=k+1$. Noticing that the weighted sum for $k+1$ is exactly the weighted sum for $k$ (which can be reformulated as desired, through our induction hypothesis) with $4$ additional inequalities, we get the following valid inequality
\begin{align*}
0\geq& \tfrac{A_k}{1+\sigma}(h(x_k)-h_\star)  -\tfrac14\normsq{x_\star-x_0} + \tfrac14\normsq{x_\star-x_0+\tfrac{2}{1+\sigma}\sum_{i=1}^{k}(A_{i}-A_{i-1})g_i} \\
 &+ \tfrac{(1-\sigma)}{2\sigma}\sum_{i=1}^k A_i\lambda_i\normsq{e_i+\tfrac{\sigma}{1+\sigma}g_i} \\
 & + \tfrac{A_{k+1}-A_k}{1+\sigma}[h(u_{k+1})-h_\star + \inner{g_{k+1}}{x_\star-u_{k+1}}]\\
 & + \tfrac{(1-\sigma)A_{k+1}}{(1+\sigma)\sigma}[h(u_{k+1})-h(x_{k+1}) + \inner{g_{k+1}}{x_{k+1}-u_{k+1}}]\\
 & + \tfrac{A_k}{1+\sigma}[h(u_{k+1})-h(x_{k}) + \inner{g_{k+1}}{x_k-u_{k+1}}]\\
 & + \tfrac{A_{k+1}}{(1+\sigma)\sigma}\left[\tfrac{\lambda_{k+1}}{2}\normsq{e_{k+1}} - \tfrac{\lambda_{k+1}\sigma^2}{2}\normsq{e_{k+1}+g_{k+1}} + h(x_{k+1}) - h(u_{k+1}) -\inner{g_{k+1}}{x_{k+1}-u_{k+1}}\right].
\end{align*}
By regrouping all function values we get the following simplification:
\begin{align*}
&[\tfrac{A_k}{1+\sigma}-\tfrac{A_k}{1+\sigma}]h(x_k)+\tfrac{A_{k+1}}{1+\sigma}[\tfrac{1}{\sigma}-\tfrac{1-\sigma}{\sigma}](h(x_{k+1})-h_\star)
+\tfrac{1}{1+\sigma}[A_{k+1}-A_k + \tfrac{1-\sigma}{\sigma}A_{k+1} + A_k - \tfrac{1}{\sigma}A_{k+1}]h(u_{k+1})\\
&=\tfrac{A_{k+1}}{1+\sigma}(h(x_{k+1})-h_\star),   
\end{align*}
where $h(x_k)$ and $h(u_{k+1})$ disappear. The remaining inequality is therefore
\begin{equation}\label{eq:intermed}
\begin{aligned}
0\geq& \tfrac{A_{k+1}}{1+\sigma}(h(x_{k+1})-h_\star) -\tfrac14\normsq{x_0-x_\star} + \tfrac14\normsq{x_\star-x_0+\tfrac{2}{1+\sigma}\sum_{i=1}^k (A_i-A_{i-1}) g_i}\\
&+ \tfrac{(1-\sigma)}{2\sigma}\sum_{i=1}^k A_i\lambda_i\normsq{e_i+\tfrac{\sigma}{1+\sigma}g_i}+
\tfrac{A_{k+1}\lambda_{k+1}}{2(1+\sigma)\sigma}[\normsq{e_{k+1}} - \sigma^2\normsq{e_{k+1}+g_{k+1}}]
\\&+ \tfrac{1}{1+\sigma}\inner{g_{k+1}}{(A_{k+1}-A_k)(x_\star-u_{k+1}) - A_{k+1}(x_{k+1}-u_{k+1})}+A_k(x_k-u_{k+1})\\
= & \tfrac{A_{k+1}}{1+\sigma}(h(x_{k+1})-h_\star) -\tfrac14\normsq{x_0-x_\star} + \tfrac14\normsq{x_\star-x_0+\tfrac{2}{1+\sigma}\sum_{i=1}^k (A_i-A_{i-1}) g_i}\\
&+ \tfrac{(1-\sigma)}{2\sigma}\sum_{i=1}^k A_i\lambda_i\normsq{e_i+\tfrac{\sigma}{1+\sigma}g_i}+
\tfrac{A_{k+1}\lambda_{k+1}}{2(1+\sigma)\sigma}[\normsq{e_{k+1}} - \sigma^2\normsq{e_{k+1}+g_{k+1}}]
\\&+ \tfrac{1}{1+\sigma}\inner{g_{k+1}}{(A_{k+1}-A_k)x_\star- A_{k+1}x_{k+1} + A_kx_k}.
\end{aligned}
\end{equation}
Then, by using \eqref{eq:Ak}, one can observe that 
\begin{align*}
A_{k+1}x_{k+1} =& \tfrac{A_{k+1}\lambda_{k+1}}{A_{k+1}-A_{k}}\left(x_0-\tfrac{2}{1+\sigma}\sum_{i=1}^{k}(A_{i}-A_{i-1}) g_i\right)+\left(A_{k+1}-\tfrac{A_{k+1}\lambda_{k+1}}{A_{k+1}-A_{k}}\right)x_{k} \\
&-A_{k+1}\lambda_{k+1} (g_{k+1}+e_{k+1})\\
=& (A_{k+1}-A_k)\left(x_0-\tfrac{2}{1+\sigma}\sum_{i=1}^{k}(A_{i}-A_{i-1}) g_i\right) + A_kx_k-A_{k+1}\lambda_{k+1} (g_{k+1}+e_{k+1}),
\end{align*}
and by re-injecting this inside the last line of \eqref{eq:intermed}, we get
\begin{align*}
0\geq & \tfrac{A_{k+1}}{1+\sigma}(h(x_{k+1})-h_\star) -\tfrac14\normsq{x_0-x_\star} + \tfrac14\normsq{x_\star-x_0+\tfrac{2}{1+\sigma}\sum_{i=1}^k (A_i-A_{i-1}) g_i}\\
&+ \tfrac{(1-\sigma)}{2\sigma}\sum_{i=1}^k A_i\lambda_i\normsq{e_i+\tfrac{\sigma}{1+\sigma}g_i}+
\tfrac{A_{k+1}\lambda_{k+1}}{2(1+\sigma)\sigma}[\normsq{e_{k+1}} - \sigma^2\normsq{e_{k+1}+g_{k+1}}]
\\&+ \tfrac{1}{1+\sigma}\inner{(A_{k+1}-A_k)g_{k+1}}{x_\star- x_0 +\tfrac{2}{1+\sigma}\sum_{i=1}^{k}(A_{i}-A_{i-1}) g_i }\\
&+ \tfrac{A_{k+1}\lambda_{k+1}}{1+\sigma}\inner{g_{k+1}}{(g_{k+1}+e_{k+1})}.
\end{align*}
We can then proceed in a similar manner as in the case $k=1$ for factorizing the quadratic terms,
\begin{align*}
0\geq & \tfrac{A_{k+1}}{1+\sigma}(h(x_{k+1})-h_\star) -\tfrac14\normsq{x_0-x_\star} + \tfrac14\normsq{x_\star-x_0+\tfrac{2}{1+\sigma}\sum_{i=1}^{k+1} (A_i-A_{i-1}) g_i}\\
&+ \tfrac{(1-\sigma)}{2\sigma}\sum_{i=1}^k A_i\lambda_i\normsq{e_i+\tfrac{\sigma}{1+\sigma}g_i}+
\tfrac{A_{k+1}\lambda_{k+1}}{2(1+\sigma)\sigma}[\normsq{e_{k+1}} - \sigma^2\normsq{e_{k+1}+g_{k+1}}]
\\&-\tfrac{(A_{k+1}-A_k)^2}{(1+\sigma)^2}\normsq{g_{k+1}}
+ \tfrac{A_{k+1}\lambda_{k+1}}{1+\sigma}\inner{g_{k+1}}{(g_{k+1}+e_{k+1})}\\
= & \tfrac{A_{k+1}}{1+\sigma}(h(x_{k+1})-h_\star) -\tfrac14\normsq{x_0-x_\star} + \tfrac14\normsq{x_\star-x_0+\tfrac{2}{1+\sigma}\sum_{i=1}^{k+1} (A_i-A_{i-1}) g_i}\\
&+ \tfrac{(1-\sigma)}{2\sigma}\sum_{i=1}^k A_i\lambda_i\normsq{e_i+\tfrac{\sigma}{1+\sigma}g_i}+
\tfrac{A_{k+1}\lambda_{k+1}(1-\sigma)}{2\sigma}\normsq{e_{k+1}+\tfrac{\sigma}{1+\sigma}g_{k+1}} 
\\&+[\tfrac{A_{k+1}\lambda_{k+1}}{(1+\sigma)}-\tfrac{A_{k+1}\lambda_{k+1}\sigma}{2(1+\sigma)}-\tfrac{(A_{k+1}-A_k)^2}{(1+\sigma)^2} - \tfrac{A_{k+1}\lambda_{k+1}\sigma(1-\sigma)}{2(1+\sigma)^2}]\normsq{g_{k+1}}
\\
=& \tfrac{A_{k+1}}{1+\sigma}(h(x_{k+1})-h_\star) -\tfrac14\normsq{x_0-x_\star} + \tfrac14\normsq{x_\star-x_0+\tfrac{2}{1+\sigma}\sum_{i=1}^{k+1} (A_i-A_{i-1}) g_i}\\
&+ \tfrac{(1-\sigma)}{2\sigma}\sum_{i=1}^{k+1} A_i\lambda_i\normsq{e_i+\tfrac{\sigma}{1+\sigma}g_i}+\tfrac{A_{k+1}\lambda_{k+1}}{(1+\sigma)}[1-\tfrac{\sigma}{2}-\tfrac{1}{(1+\sigma)} - \tfrac{\sigma(1-\sigma)}{2(1+\sigma)}]\normsq{g_{k+1}}\\
= & \tfrac{A_{k+1}}{1+\sigma}(h(x_{k+1})-h_\star) -\tfrac14\normsq{x_0-x_\star} + \tfrac14\normsq{x_\star-x_0+\tfrac{2}{1+\sigma}\sum_{i=1}^{k+1} (A_i-A_{i-1}) g_i}\\
&+ \tfrac{(1-\sigma)}{2\sigma}\sum_{i=1}^{k+1} A_i\lambda_i\normsq{e_i+\tfrac{\sigma}{1+\sigma}g_i},
\end{align*}since $1-\tfrac{\sigma}{2}-\tfrac{1}{1+\sigma} - \tfrac{\sigma(1-\sigma)}{2(1+\sigma)} = 0$ and this concludes the proof.
\end{proof}

\section{Tightness of Theorem~\ref{thm:orip2}}\label{sec:tight_orip}
\begin{proof}One can verify that the guarantee for \eqref{eq:orip} provided by Theorem~\ref{thm:orip2} is actually non-improvable. That is, for all $\{\lambda_k\}_k$ with $\lambda_k>0$, $\sigma\in[0,1]$, $d\in\mathbb{N}$, $x_0\in\mathbb{R}^d$, and $N\in\mathbb{N}$, there exists $f\in\Fccp(\Rd)$ such that this bound is achieved with equality. For proving this statement, it is sufficient to exhibit a one-dimensional function for which the bound is attained, which is what we do below. The bound is attained on the one-dimensional linear minimization problem
\begin{equation}
\min_x\, \{f(x)\equiv c\, x + i_{\R_+}(x)\},\label{eq:linoptim}
\end{equation}
with an appropriate choice of $c> 0$, where $i_{\R_+}$ denotes the convex indicator function of $\R_+$. Indeed, one can check that the relative error criterion
\[\exists u_{k}\in\R_+,\;\tfrac{\lambda_k}{2}\normsq{e_k} + f(x_k)-f(u_k) -\inner{g_k}{x_k-u_k}\leq \tfrac{\lambda_k\sigma^2}{2}\normsq{e_k+g_{k}} \]
is satisfied with equality when picking $g_{k}=c$ ($g_k$ is thus a subgradient at $x_k$), $u_k=x_k$, and $e_{k}=-\tfrac{c\sigma}{1+\sigma}$; and hence $x_{k}=y_{k-1}-\tfrac{c\lambda_k}{1+\sigma}$. The argument is then as follows: if for some $x_0>0$ and $0\leq h\leq x_0/c$ we manage to show that $x_N=x_0-c h$, then $f(x_N)-f(x_\star)=c(x_0-c h)$ and hence the value of $c$ producing the worst possible (maximal) value of $f(x_N)$ is $c=\tfrac{x_0}{2h}$. In that case, the resulting value is $f(x_N)-f(x_\star)=\tfrac{x_0^2}{4h}$. Therefore, in order to prove that the guarantee from Theorem~\ref{thm:orip2} cannot be improved, we show that $x_N=x_0-\tfrac{A_N}{1+\sigma}c$ on the linear problem~\eqref{eq:linoptim}. It is easy to show that $x_1=x_0-\tfrac{A_1}{1+\sigma}c$ using $A_1=\lambda_1$. The argument follows by induction: assuming $x_{k}=x_0-\tfrac{A_k}{1+\sigma}c$, one can compute
\begin{equation*}
\begin{aligned}
x_{k+1} =& \tfrac{\lambda_{k+1}}{A_{k+1}-A_{k}}\left(x_0-\tfrac{2}{1+\sigma}\sum_{i=1}^{k}(A_{i}-A_{i-1}) g_i\right)+\left(1-\tfrac{\lambda_{k+1}}{A_{k+1}-A_{k}}\right)x_{k} \\&-\lambda_{k+1} (g_{k+1}+e_{k+1})\\
=& \tfrac{\lambda_{k+1}}{A_{k+1}-A_{k}}\left(x_0-\tfrac{2c}{1+\sigma}A_{k}\right)+\left(1-\tfrac{\lambda_{k+1}}{A_{k+1}-A_{k}}\right)\left( x_0-\tfrac{A_k}{1+\sigma}c\right) -\lambda_{k+1} \tfrac{c}{1+\sigma}\\
=&x_0-\tfrac{c}{1+\sigma}\tfrac{2\lambda_{k+1}A_k + (A_{k+1}-A_k)A_k - \lambda_{k+1}A_k + \lambda_{k+1}(A_{k+1}-A_k)}{A_{k+1}-A_k}\\
=&x_0-\tfrac{c}{1+\sigma}\tfrac{ (A_{k+1}-A_k)A_k  + \lambda_{k+1}A_{k+1}}{A_{k+1}-A_k}\\
=&x_0-\tfrac{c}{1+\sigma}A_{k+1},
\end{aligned}
\end{equation*}
where the second equality follows from simple substitutions, and the last equalities follow from basic algebra and $\lambda_{k+1}A_{k+1} = (A_{k+1}-A_k)^2$. The desired statement is proved by picking $c=\tfrac{(1+\sigma)x_0}{2A_N}$, reaching $f(x_N)-f(x_\star) = \tfrac{(1+\sigma)x_0^2}{A_N}$.
\end{proof}
\section{Tightness of Theorem~\ref{thm:ppa-str}}\label{app:tight-str}
\begin{proof}
We show that the guarantee provided in Theorem~\ref{thm:ppa-str} is non-improvable. That is, for all $\mu \geq 0$, $\{\lambda_k\}_k$ with $\lambda_k\geq0$, $\sigma\in[0,1]$, $d\in\mathbb{N}$, $w_0\in\mathbb{R}^d$, and $N\in\mathbb{N}^*$, there exists $h\in\Fmu(\Rd)$ such that this bound is achieved with equality. Indeed, the bound is attained on the simple quadratic minimization problem 
\begin{equation}
    \min_x \{h(x) \equiv \tfrac{\mu}{2}\normsq{x}\}.
\end{equation}
We can check that the relative error criterion
\[\tfrac{\lambda_k}{2}\normsq{e_k} \leq \tfrac{\sigma^2\lambda_k}{2}\normsq{e_k+v_k}, \]
is satisfied with equality when picking $v_k = \nabla h(w_{k+1}) = \mu w_{k+1}$ and $e_k = -\tfrac{\sigma}{1+\sigma}v_k$. Under these choices, one can write 
\[w_{k+1} = w_k -\tfrac{\lambda_{k+1}\mu}{1+\sigma}w_{k+1}, \]
which leads to 
\[w_{k+1} = \tfrac{1+\sigma}{1+\sigma + \lambda_{k+1}}w_k. \]
Finally 
\[ w_{N} = \prod_{i=1}^{N}\tfrac{1+\sigma}{1+\sigma+\lambda_{i}\mu}w_0,\]
and the desired results follows.
\end{proof}

\end{document}